\documentclass[11pt]{article} 
\usepackage{amsmath}
\usepackage{amssymb}
\usepackage{amsthm,mathtools}
\usepackage{lineno}
\usepackage[mathscr]{eucal}
\usepackage{xcolor}
\usepackage{fontenc}
\usepackage{graphicx}
\usepackage{geometry}
\usepackage{lineno}
\theoremstyle{definition}
\newtheorem{definition}{Definition}[section]
\newtheorem{theorem}[definition]{Theorem}
\newtheorem*{theorem*}{Conjecture}

\newtheorem{lemma}[definition]{Lemma}
\newtheorem{corollary}[definition]{Corollary}
\theoremstyle{remark}
\newtheorem{remark}[definition]{Remark}
\newtheorem{example}[definition]{Example}
\newcounter{enumctr}

\begin{document}
\title{The oscillatory solutions of multi-order fractional differential equations}
\author{Ha Duc Thai\footnote{\tt hdthai@math.ac.vn, \rm Institute of Mathematics, Vietnam Academy of Science and Technology, 18 Hoang Quoc Viet, 10307 Ha Noi, Vietnam},
	Hoang The Tuan\footnote{\tt httuan@math.ac.vn, \rm Institute of Mathematics, Vietnam Academy of Science and Technology, 18 Hoang Quoc Viet, 10307 Ha Noi, Viet Nam}}
 \date{} 
\maketitle 
	\begin{abstract} 
This paper systematically treats the asymptotic behavior of many (linear/nonlinear) classes of higher-order fractional differential equations with multiple terms. To do this, we utilize the characteristics of Caputo fractional differentiable functions, the comparison principle, counterfactual reasoning, and the spectral analysis method (concerning the integral presentations of basic solutions). Some numerical examples are also provided to demonstrate the validity of the proposed results.
	\end{abstract}
{{\bf Key works}: Fractional differential equations, multi-order fractional differential equations with Emden--Fowler-type coefficients, oscillatory solutions, non-oscillatory solution, characteristic polynomials, comparison principle, asymptotic behavior of solutions}\\


\section{Introduction}
J.C.F. Sturm initiated oscillation theory in his investigations of Sturm--Louville problems in 1836. Due to its great importance in describing many real applications, many papers dealing with non-autonomous ordinary differential equations have appeared, in which various classifications of equations according to the oscillatory properties of their solutions are proposed. Furthermore, the existence or absence of singular, proper, oscillatory, and monotone solutions of various types are shown and the asymptotic properties of such solutions are considered. We refer in particular to the survey paper by Wong \cite{Wong}, to the monographs by Elias \cite{Elias}, Kiguradze and Chanturiya \cite{KC}, Swanson \cite{Swanson} and the references therein for some representative contributions related to this topic.

Since the Leibnitz rule for derivatives of composite functions is not valid for fractional derivatives, many approaches and methods in ordinary differential equations are often not applicable to fractional differential equations concerning the study of the oscillatory properties of solutions. Hence, only a few studies on this theory have been published, and the development is still in its infancy and requires further investigation. From a personal perspective, we list some typical works below.

Grace \cite{Grace1} reported the first results on oscillatory solutions of fractional differential equations. In particular, in this paper, using a counterfactual argument when dealing with the definition of non-oscillation, the author obtained some simple sufficient conditions ensuring the existence of oscillatory solutions for some basic classes of equations with Riemann--Liouville derivatives of the order less than 1. After that in \cite{Grace1}, he showed the oscillatory behavior of solutions to nonlinear fractional differential equations with the Caputo fractional derivative of the order $\alpha\in (1,2)$. Băleanu et al. \cite{Baleanu} investigated the eventual sign-changing for the solutions of a linear $(1+\alpha)$-order fractional differential equation ($\alpha\in (0,1)$) by presenting a Kamenev-type theorem in the framework of fractional calculus. A survey on one of the mathematical approaches used to solve a fractional differential equation whose solution gives the free dynamic response of viscoelastic single degree of freedom systems (viscosity is modeled by a fractional displacement derivative instead of first-order one) is introduced in \cite{Marano}. The paper dealt with the Caputo fractional derivative and its Laplace transform (on which the resolution method is based). 

Recently in \cite{Dosla}, the authors considered higher-order fractional differential equations with the Riemann--Liouville fractional derivatives and Emden--Fowler-type coefficients. They explored the effect of different orders of derivatives on the oscillatory and asymptotic properties of the studied equations. Moreover, the dissimilarities between integer and non-integer order cases are emphasized.

Inspired by the works of Grace, Bartušek, and Došlá mentioned above, we focus on multi-term fractional differential equations with the Caputo fractional derivatives in the present paper. In light of the characteristics of Caputo fractional differentiable functions, the comparison principle, counterfactual reasoning, and the spectral analysis method (concerning the integral presentations of basic solutions), we systematically treat the asymptotic behavior of many (linear/nonlinear) classes of higher-order fractional differential equations with multiple terms. More precisely, in Section 3, we discuss the existence of (linear/nonlinear) fractional differential equations. The asymptotic behavior of oscillatory solutions is described in Section 4. The asymptotic behavior of non-oscillatory solutions of higher-order fractional differential equations with the Emden--Fowler-type coefficients is given in Section 5. We also provide numerical examples to illustrate the obtained theoretical findings.
\section{Preliminaries}
Let $0 <\alpha \in \mathbb R$ with $\left \lceil \alpha\right \rceil = n $ and $J = [0, T]$ or $J = [0, \infty)$ we defined Riemann-Liouville fractional integral of a function $f: J \rightarrow 
 \mathbb R$ as below:
 $$ I^\alpha_{0^+}f(t) := \frac{1}{\Gamma(\alpha)}\int_{0}^{t}(t-s)^{\alpha -1}f(s)ds,\;t\in J,$$
 its Riemann-Liouville fractional derivative order $\alpha$ of $f$ as 
 $$^{RL} D^\alpha_{0^+} f(t) = D^{(n)} I_{0^+}^{n-\alpha}f(t), \; t \in J \setminus \{ 0 \}, $$
 and its Caputo fractional derivative order $\alpha$
 of function $f$ as below:
 $$ ^C D^\alpha_{0^+}f(t) :=^{RL} D^\alpha_{0^+}(f(t) -T_{n-1}[f;0](t)) , \;t\in J \setminus \{ 0 \},$$
 where $T_{n-1}[f;0](t) = \sum_{k=0}^{n-1}\frac{f^{(k)}(0)}{k!}t^k$ denotes the Taylor polynomial of order
$n - 1$ of $f$ centered at $0$ and $D^{(n)}$ is usually derivative. See 
 \cite[Chapter III]{Kai} and \cite{Vainikko_16} for more details on the Caputo fractional derivative. 
\begin{lemma}\label{nhtp}
    Let $\alpha, \beta > 0$ and $x \in L^1[0,T]$. Then, we have 
    \begin{align*}
        I^\alpha_{0^+}\left(I^\beta_{0^+}x(t)\right) =  I^\beta_{0^+}\left(I^\alpha_{0^+}x(t)\right) = I^{\alpha+\beta}_{0^+}x(t),\;\forall t \in [0,T].
    \end{align*}
\end{lemma}
\begin{proof}
    See \cite[Theorem 2.2 and Corollary 2.3 page 14]{Kai}
\end{proof}
\begin{lemma}\label{dhc}
    Let $f \in AC[0,T], \alpha \in (0,1)$. Then,
    \begin{align*}
        ^C D^\alpha_{0^+}f(t) = \frac{1}{\Gamma(1-\alpha)}\int_0^t \frac{f'(s)}{(t-s)^\alpha}\mathrm{d}s,\;\forall t \in (0,T).
    \end{align*}
\end{lemma}
 \begin{proof}
     See \cite[Lemma 2.12, p. 27 and Lemma 3.4, p. 53]{Kai}.
 \end{proof}
 \begin{lemma}\label{tddh}
     If $f$ is a continuous and $\alpha \in (0,1)$, then
     \begin{align*}
         ^C D^\alpha_{0^+} I^\alpha_{0+}f =f.
     \end{align*}
 \end{lemma}
 \begin{proof}
     See \cite[Theorem3.7, p. 53]{Kai}.
 \end{proof}
  \begin{definition}\cite[Definition 2.2]{Cong}
    Let $0 <\alpha \in \mathbb R$ with $\left \lceil \alpha\right \rceil = n $ and $f \in C^{n-1}[0,T]$. If $^C D^\alpha_{0^+} f$ exists and is in the class $L_1[0,T]$, then we say that the function $f$ is Caputo $\alpha$-differentiable on $[0,T]$. If $^C D^\alpha_{0^+} f$ exists and is in the class $C[0,T]$, then we say that the function $f$ is continuously Caputo $\alpha$-differentiable on $[0,T]$.
 \end{definition}
 \begin{lemma}\label{tdtp}
     Let $\alpha > 0$. Suppose that $x$ is continuously Caputo $\alpha$-differentiable on $[0,\infty)$, then 
     \begin{align*}
         I^\alpha_{0^+}\left( ^C D^\alpha_{0^+}x(t) \right) = x(t) - \sum_{k=0}^{\left \lfloor \alpha \right \rfloor-1} \frac{x^{(k)}(0)}{k!}t^k,\;\forall t \in [0,\infty).
     \end{align*}
 \end{lemma}
 \begin{proof}
     See \cite[Proposition 5.1]{Vainikko_16}.
 \end{proof}
 \begin{lemma}\label{bdv}
     Let $n-1 <\alpha < n, n\in\mathbb N $ and assume that $x :[0,\infty) \to \mathbb R$ is continuously Caputo $\alpha$-differentiable on $[0,\infty)$. Then, for any $t \in (0, \infty)$, we have 
     \begin{align*}
         ^C D^\alpha_{0^+}x(t) = \frac{1}{\Gamma( n-\alpha)}\left(\frac{x^{(n-1)}(t) - x^{(n-1)}(0)}{t^{\alpha-n+1}}\right) + (\alpha-n+1)\int_0^t \frac{x^{(n-1)}(t) - x^{(n-1)}(s)}{(t-s)^{\alpha-n+2}}ds.
     \end{align*}
 \end{lemma}
 \begin{proof}
     See \cite[Theorem 5.2]{Vainikko_16}.
 \end{proof}
\begin{lemma}\label{bdt}\cite[Lemma 2.1]{Grace1}
 For $X \geq 0$ and $Y > 0$, we have 
 \begin{itemize}
     \item [(i)] $\lambda X Y^{\lambda - 1} - X^\lambda \leq (\lambda -1)Y^\lambda, \lambda > 1.$
     \item [(ii)]$\lambda X Y^{\lambda - 1} - X^\lambda \geq (  \lambda-1) Y^\lambda, 0 < \lambda <1.$ 
 \end{itemize}
 \end{lemma}
 \begin{definition}
Let $x:[0,\infty)\rightarrow \mathbb R$. It is said to be eventually (non) negative, positive if there is a $T>0$ such that $x(t)<0$ ($>0$), $x(t)>0$ for all $t\geq T$, respectively.
 \end{definition}
 Consider the multi-term fractional differential equation
\begin{align}\label{2.6}
    \sum_{i=1}^{k-1}a_i\, ^C D^{\alpha_i}_{0^+}x(t) +  ^C D^{\alpha_k}_{0^+}x(t) =f(t,x),\; t > 0
\end{align}
with the initial conditions
\begin{align}\label{2.7}
    x(0) = x_0, x^{(j)}(0) = x_j, j = 1,2,\ldots,k-1,
\end{align}
where $0 < \alpha_1 \leq 1 <  \alpha_{2}\leq 2 < \ldots <\alpha_k \leq k $, $a_i, i=1,2,\ldots,k$ are constants, and $f : [0,\infty)\times \mathbb R \to \mathbb R$ is continuous.
\begin{definition}
    A function $x: [0,\infty) \to \mathbb R$ is called a solution of the system \eqref{2.6}--\eqref{2.7} if it is continuously Caputo $\alpha$-differentiable on $[0,\infty)$ and satisfies \eqref{2.6} on $(0,\infty)$ and the initial conditions \eqref{2.7}.
\end{definition}
\begin{definition}
	A solution $x$ of the system \eqref{2.6}--\eqref{2.7} is oscillatory if it exists globally on $[0,\infty)$ and there is a sequence $\{t_n\}_{n=1}^\infty\subset [0,\infty)$ with $t_n\to\infty$ as $n\to\infty$ such that $x(t_n)=0$, $n\in \mathbb N$. Otherwise, it is called non-oscillatory.
	\end{definition}
\begin{theorem}\label{dlttdn}
     Consider the system \eqref{2.6}--\eqref{2.7}. Suppose that $f(\cdot,\cdot)$ is Lipchitz continuous to the second variable, that is, there is a continuous function $L : [0,\infty) \to [0, \infty)$ such that
    \begin{align*}
        |f(t,x) -f(t,\hat x)| \leq L(t)|x-\hat x|, \quad\forall t \in [0,\infty), x, \hat x\in \mathbb R.
    \end{align*}
 Then, for any $x_j\in \mathbb R$, $j=0,1,\dots,k-1$, the system \eqref{2.6}--\eqref{2.7} has a unique solution on $[0,\infty).$
\end{theorem}
\begin{proof}
    Using the same arguments as in the proof of \cite[Theorem 4.5]{Cong}.
\end{proof}
 \section{On the existence of the oscillatory solutions of fractional differential equations}
\subsection{The oscillation of multi-order fractional differential equations}
 We first focus on the following fractional differential equation
 \begin{align}\label{eq4}
^C D^\alpha_{0^+}x(t) + a\,\,{}^C D^\beta_{0^+}x(t) = f(t,x(t)) + g(t),\;\;t>0
\end{align}
 with the initial conditions
 \begin{align}\label{dkd1}
 x(0) = x_0, x^{(1)}(0) = x_1, \ldots, x^{(n-1)}(0) = x_{n-1},
 \end{align}
 where $0 < \beta < \alpha \leq  \lceil  \alpha \rceil = n\in\mathbb N$, $x_0, x_1, \ldots, x_{n-1} \in \mathbb R, $ $a$ is a non-negative real number, $f : [0, \infty) \times \mathbb R \rightarrow \mathbb R$ and $g : [0, \infty) \rightarrow \mathbb R$ are continuous functions satisfying the assumptions below.
 \begin{itemize}
     \item [(A)] There exists $T > 0$ such that
     $x f(t,x) \leq 0 $ for all $t \geq T$ and $x \in \mathbb R$.
     \item [(B)] \begin{align*}
&\limsup_{t \to +\infty}t^{\beta-\alpha-n+1}\int_0^t (t-s)^{\alpha-1}g(s)ds = +\infty,\\
&\liminf_{t \to +\infty}t^{\beta-\alpha-n+1}\int_0^t (t-s)^{\alpha-1}g(s)ds = -\infty.
\end{align*} 
 \end{itemize}
Let $m:= \lceil  \beta \rceil$. If $\beta \notin \mathbb Z$, we set $\alpha_1 = \beta -[\beta]+1$, $\alpha_i = \alpha_1 + i-1, i=2, 3, \ldots, m+1$, $\alpha_{i} =i-1, i = m+2, \ldots, n-1$, $\alpha_n = \alpha$. If $\beta \in \mathbb Z$, we set $\alpha_i = i, i=1,2, \ldots, n-1$, $\alpha_n = \alpha$. Then the equation \eqref{eq4} is rewritten as 
\begin{equation}\label{ptvl}
    ^C D^{\alpha_n}_{0^+}x(t) = F(t,x(t), ^C D^{\alpha_1}_{0^+}x(t), ^C D^{\alpha_2}_{0^+}x(t), \ldots, ^C D^{\alpha_{n-1}}_{0^+}x(t)),\;t>0,
\end{equation}
here 
\begin{align*}
    F(t,x(t), ^C D^{\alpha_1}_{0^+}x(t), \ldots, ^C D^{\alpha_{n-1}}_{0^+}x(t)) = \begin{cases}
 f(t,x(t)) +g(t) - a\,^CD^{\alpha_{m+1}}_{0^+}x(t)  & \text{ if } \beta \notin \mathbb Z, \\ 
 f(t,x(t)) +g(t) - a\,^CD^{\alpha_{m}}_{0^+}x(t)    & \text{ if } \beta \in \mathbb Z. 
\end{cases}
\end{align*}
By \cite[Theorem 17]{Cong}, if $f$ is globally Lipschitz continuous concerning the second variable, the system \eqref{ptvl}--\eqref{dkd1} has a unique global solution on $[0,\infty).$
\begin{theorem}\label{dl1}
     Assume that the conditions \textup{(A)} and \textup{(B)} hold. Then an arbitrary solution of \eqref{eq4} (if it exists globally) is oscillatory.
\end{theorem}
\begin{proof}
	The theorem will be proven by contradiction. Suppose that $x$ is a non-oscillatory solution of the system \eqref{eq4}--\eqref{dkd1}. Without loss of generality, assume that $x(t) > 0$ for $t\geq t_1$. According to the assumption \textup{(A)}, we can find $t_2>\max\{T,t_1\}$ so that $x(t)f(t,x(t) \leq 0$ for all $t \geq t_2$ and thus $f(t,x(t)) \leq 0$ for all $t \geq t_2.$ On the other hand, for all $t > 0$, we have
   \begin{align*}
       I^\alpha_{0^+}\left(^C D^\alpha_{0^+}x(t) + a\,\,{}^C D^\beta_{0^+}x(t)\right) &= x(t) - \sum_{k=0}^{n-1}\frac{x^{(k)}(0)}{k!}t^k + a I^{\alpha-\beta}_{0^+}\left(x(t) - \sum_{k=0}^{m-1}\frac{x^{(k)}(0)}{k!}t^k\right) \\
       &\hspace{-2cm}=x(t) - \sum_{k=0}^{n-1}\frac{x^{(k)}(0)}{k!}t^k - \sum_{k=0}^{m-1}\frac{ax^{(k)}(0)}{\Gamma(\alpha-\beta+1+k)}t^{\alpha-\beta+k} + aI^{\alpha-\beta}_{0^+}x(t),
   \end{align*}
   where $m= \lceil  \beta \rceil$. From this,
    \begin{align*}
        x(t)= &\sum_{k=0}^{n-1}\frac{x^{(k)}(0)}{k!}t^k + \sum_{k=0}^{m-1}\frac{ax^{(k)}(0)}{\Gamma(\alpha-\beta+1+k)}t^{\alpha-\beta+k} -\frac{a}{\Gamma(\alpha-\beta)}\int_0^t(t-s)^{\alpha-\beta-1}x(s)ds\\
        &+\frac{1}{\Gamma(\alpha)}\int_0^t(t-s)^{\alpha-1}f(s,x(s))ds +  \frac{1}{\Gamma(\alpha)}\int_0^t(t-s)^{\alpha-1}g(s)ds,\;\;\forall t>0.
    \end{align*} 
    Put $b = \max\{|x_0|,|x_1|,\ldots,|x_{n-1}|\}$, $M_1 = \max_{s\in [0,t_2]}|x(s)|$, $M_2 = \max_{s\in [0,t_2]}|f(s,x(s))|$. Since $0<\beta < \alpha$, $a \geq 0 $, $x(s) > 0$, $f(s, x(s)) \leq 0$ for all $s\ \geq t_2$, we obtain
    \begin{align*}
        x(t)\leq &\sum_{k=0}^{n-1} \frac{b}{k!}t^k + \sum_{k=0}^{m-1}\frac{ab}{\Gamma(\alpha-\beta+1+k)}t^{\alpha-\beta+k} + \frac{aM_1}{\Gamma(\alpha-\beta+1)}t^{\alpha-\beta}\\
        &+\frac{1}{\Gamma(\alpha)}\int_0^{t_2}(t-s)^{\alpha-1}f(s;x(s))ds +\frac{1}{\Gamma(\alpha)}\int_0^t(t-s)^{\alpha-1}g(s)ds,\;\;\forall t>t_2.
    \end{align*}
    This implies that for all $t>t_2$,
    \begin{align}\label{p1}
        t^{\beta-\alpha-n+1}x(t)\leq &\sum_{k=0}^{n-1} \frac{b}{k!}t^{k+\beta-\alpha-n+1} + \sum_{k=0}^{m-1}\frac{ab}{\Gamma(\alpha-\beta+1+k)}t^{k-n+1} + \frac{aM_1}{\Gamma(\alpha-\beta+1)}t^{1-n}\nonumber\\
        &\hspace{-2cm}+\frac{t^{\beta-n+1-\alpha}}{\Gamma(\alpha)}\int_0^{t_2}(t-s)^{\alpha-1}f(s,x(s))ds +\frac{t^{\beta-\alpha-n+1}}{\Gamma(\alpha)}\int_0^t(t-s)^{\alpha-1}g(s)ds.
    \end{align}
    Consider the case when $\alpha \leq 1$. Notice that for all $s \in [0,t_2]$, $t>t_2$, $t(t_2-s)\leq t_2(t-s)$. It deduces that $\frac{t}{t-s}\leq \frac{t_2}{t_2-s}$. Hence, 
    \begin{align}\label{p2}
        \frac{t^{1-\alpha}}{\Gamma(\alpha)}\int_0^{t_2}(t-s)^{\alpha-1}f(s,x(s))ds &\leq \frac{M_2 t^{1-\alpha}}{\Gamma(\alpha)}\int_0^{t_2}(t-s)^{\alpha-1}ds \nonumber \\
        &=\frac{M_2}{\Gamma(\alpha)}\int_0^{t_2}\left ( \frac{t}{t-s} \right )^{1-\alpha}ds\nonumber \\
        &\leq \frac{M_2}{\Gamma(\alpha)}\int_0^{t_2}\left ( \frac{t_2}{t_2-s} \right )^{1-\alpha}ds\nonumber \\
        &=\frac{M_2 t_2}{\Gamma(\alpha+1)},\;\;\forall t>t_2.
    \end{align}
    To deal with $\alpha > 1$, we see that
     \begin{align}\label{p3}
        \frac{t^{1-\alpha}}{\Gamma(\alpha)}\int_0^{t_2}(t-s)^{\alpha-1}f(s,x(s))ds &\leq \frac{M_2 t^{1-\alpha}}{\Gamma(\alpha)}\int_0^{t_2}(t-s)^{\alpha-1}ds \nonumber \\
        &=\frac{M_2}{\Gamma(\alpha)}\int_0^{t_2}\left ( \frac{t-s}{t} \right )^{\alpha-1}ds\nonumber \\
        &\leq \frac{M_2}{\Gamma(\alpha)}\int_0^{t_2}ds\nonumber \\
        &=\frac{M_2 t_2}{\Gamma(\alpha)},\;\;\forall t>t_2.
    \end{align}
   Since $k+\beta-\alpha-n+1 < 0$ for all $k=0,\ldots,n-1$, $k-n+1 \leq 0$ for all $k=0,\ldots,m-1$ and $1-n \leq 0$, by combining \eqref{p1}, \eqref{p2} and \eqref{p3}, there is a positive constant $c =c(t_2)$ such that
     \begin{align}
        t^{\beta-\alpha-n+1}x(t) \leq c + \frac{t^{\beta-\alpha-n+1}}{\Gamma(\alpha)}\int_0^t(t-s)^{\alpha-1}g(s)ds, \quad \forall t\geq t_2,
    \end{align}
    which together with (B) leads to
    \begin{align*}
        \liminf_{t \to \infty}\, t^{\beta - \alpha-n+1}x(t) \leq c + \liminf_{t \to \infty}\, \frac{t^{\beta - \alpha-n+1}}{\Gamma(\alpha)}\int_0^t(t-s)^{\alpha-1}g(s)ds = -\infty,
    \end{align*}
    contrary to the counterfactual hypothesis that $x(t) > 0$ for all $t> t_2.$ The proof is complete.
\end{proof}
\begin{corollary}\label{hq1}
    Consider the system \eqref{eq4}--\eqref{dkd1}. Suppose that $f$ satisfies the assumption (A) in Theorem \ref{dl1} and  $g$ has the following form.
    \begin{itemize}
        \item [(B)'] $g(t) = ^C D^\alpha_{0^+} (t+1)^{\sigma}h(t)$, $t\geq 0$, where $\sigma > \alpha-\beta+n-1 > 0$, $h$ is continuously differentiable on $[0,\infty)$, $\limsup_{t\to \infty}h(t) > 0$ and $\liminf_{t\to \infty}h(t) < 0$.
    \end{itemize}
    Then, its solution (if it exists globally) is oscillatory.   
\end{corollary}
\begin{proof}
     It is sufficient to check that $g$ verifies the condition (B) in Theorem \ref{dl1}. Indeed, due to $h$ is continuously differentiable on $[0,\infty)$, according to Lemma \ref{tdtp}, we obtain
        \begin{align*}
            t^{\beta-\alpha-n+1}\int_0^t (t-s)^{\alpha-1}g(s)ds &= \Gamma (\alpha)  t^{\beta-\alpha-n+1} I^\alpha_{0^+}g(t) \\
            &= \Gamma (\alpha)  t^{\beta-\alpha-n+1}(t+1)^{\sigma}h(t), \quad t > 0.
        \end{align*}
        Notice that $\sigma > \alpha-\beta+n-1$, therefore
        \begin{align*}
            \limsup_{t \to \infty} t^{\beta-\alpha-n+1}\int_0^t (t-s)^{\alpha-1}g(s)ds = \lim_{t \to \infty}t^{ \beta -\alpha-n+1 }(t+1)^{\sigma}\limsup_{t \to \infty}h(t) = +\infty.
        \end{align*}
        In the same manner, we can see that
        \begin{align*}
            \liminf_{t \to \infty} t^{\beta-\alpha-n+1}\int_0^t (t-s)^{\alpha-1}g(s)ds = -\infty,
        \end{align*}
        which completes the proof.
\end{proof}
\begin{example}
Consider the fractional differential equation \eqref{eq4} with $\alpha = 1/2, \beta = 1/3$, $a=2$, $f(t,x) = -(t-1)x -(t^2-3)x^3$ and $g(t) = ^C D^{1/2}_{0^+}\left(t^{7/6}\sin t\right)$. This equation is rewritten as
 \begin{align}\label{vd}
^C D^{1/2}_{0^+}x(t) + 2\,\,{}^C D^{1/3}_{0^+}x(t) = -(t-1)x(t) -(t^2-3)x^3(t) + ^C D^{1/2}_{0^+}(t^{7/6}\sin t),\;t>0.
\end{align}
 It is easy to check that $xf(t,x) \leq 0$ for all $t \geq 2$. Thus $f$ satisfies the condition $\textup{(A)}$. Moreover, for all $t>0$,
 \begin{align*}
     t^{-1/6}\int_0^t (t-s)^{-1/2}g(s)ds &= \Gamma(1/2) t^{-1/6} I^{1/2}_{0^+}g(t) = \Gamma(1/2) t^{-1/6} I^{1/2}_{0^+}\left(^C D^{1/2}_{0^+}\left(t^{7/6}\sin t\right)\right)\\
     & = \Gamma(1/2) t^{-1/6} t^{7/6}\sin t = \Gamma(1/2)t\sin t.
 \end{align*}
This implies that 
\begin{align*}
    &\limsup_{t\to\infty}t^{-1/6}\int_0^t (t-s)^{-1/2}g(s)ds = +\infty, \\
    &\liminf_{t \to \infty} t^{-1/6}\int_0^t (t-s)^{-1/2}g(s)ds = -\infty.
\end{align*}
Thus the condition $\textup{(B)}$ is verified. By Theorem \ref{dl1}, all solutions of the equation \eqref{vd} are oscillatory. In Figure 1, we simulate the orbit of the solution with the initial condition $x(0)=1$ on the interval $[0,70]$. 
\begin{figure}
		\begin{center}
			\includegraphics[scale=.7]{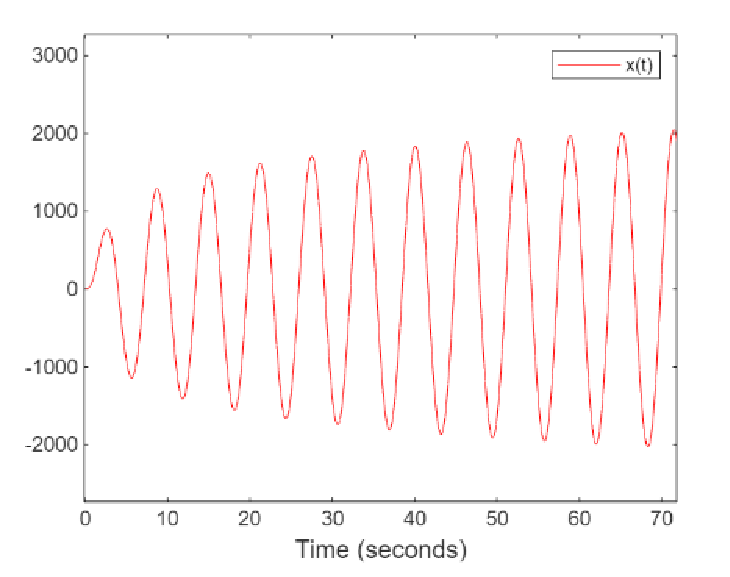}
		\end{center}
		\begin{center}
			\caption{The orbit of the solution to equation \eqref{vd} with the initial condition $x_0 = 1$ on the interval $[0,70]$.}
		\end{center}
\end{figure}
\end{example}
\begin{remark}
   Under assumptions (A) and (B), the conclusion of Theorem \ref{dl1} is still true for the following general multi-order fractional system
     \begin{equation*}
       ^C D^{\alpha_k}_{0^+}x(t)+ \sum_{i=1}^{k-1} a_i\, ^C D^{\alpha_i}_{0^+}x(t)  
        = f(t,x(t)) +g(t),\;t>0
    \end{equation*}
    with arbitrary initial values
    \begin{equation*}
 x(0) = x_0, x^{(1)}(0) = x_1, \ldots, x^{(n-1)}(0) = x_{n-1},
 \end{equation*}
    here $0 < \alpha_1 < \alpha_2 < \ldots \alpha_k \leq \left \lceil \alpha_k \right \rceil =n\in \mathbb N$, $a_i,\; i=1,\ldots,k-1,$ are nonnegative real numbers.
\end{remark}
\subsection{The oscillation of equations with higher-order nonlinearities} 
We now study the multi-term fractional differential equation with higher-order nonlinearities
\begin{align}\label{eq5}
 ^CD^\alpha_{0^+}x(t) + a\,\,^CD^\beta_{0^+}x(t)+p(t)x(t) + \sum_{i=1}^{k}
q_i(t) \mathrm{sgn}\left(x(t)\right)|x(t)|^{\lambda_i} =g(t), \qquad t > 0
\end{align} 
with the initial conditions 
\begin{align}\label{dkd2}
 x(0) = x_0, x^{(1)}(0) = x_1, \ldots, x^{(n-1)}(0) = x_{n-1},
 \end{align}
 where $0 < \beta < \alpha \leq  \lceil  \alpha \rceil = n\in\mathbb N$, $x_0, x_1, \ldots, x_{n-1} \in \mathbb R, $ $a$ is a nonnegative real number, $ 0<\lambda_1 < \lambda_2 < \cdots < \lambda_k,\; k \geq 2$ are positive real numbers, $p : [0, \infty) \rightarrow \mathbb R$, $q_i :[0, \infty) \rightarrow \mathbb R,\; i=1,2,\ldots,k,$ $\mathrm{sgn}(\cdot)$ is the sign function, and $g : [0, \infty) \rightarrow \mathbb R$ are continuous functions. Suppose that the following conditions are true.
\begin{itemize}
    \item [(C)] $\lambda_i \neq 1$ for all $i =1, 2, \ldots,k$  and satisfy one of the following two conditions
    \begin{itemize}
        \item [$(\textup{C})_1$] $1 < \lambda_1 < \lambda_2 < \cdots < \lambda_k$.
        \item [$(\textup{C})_2$] There exists $l \in \{1,\ldots,k-1\}$ such that $\lambda_1<\cdots<\lambda_l < 1 < \lambda_{l+1}<\cdots<\lambda_k$.
    \end{itemize}
    \item [(D)] There exists $T_1 \geq 0$ such that $p(t) \neq 0$ for all $t \geq T_1.$
    \item [(E)] There exists $T_2 \geq 0$ such that
    $\mathrm{sgn}(\lambda_i-1)q_i(t) >0 , i = 1,2, \ldots,m$ for all $t \geq T_2$.
    \item [(F)] There exists $T \geq \max\{T_1, T_2\}$ such that
    \begin{align*}
        &\liminf_{t \to \infty}t^{\beta-\alpha-n+1}\int_T^t (t-s)^{\alpha-1}\left(\gamma\sum_{i=1}^k  |p(s)|^{\frac{\lambda_i}{\lambda_i-1}}|q_i(s)|^{\frac{1}{1-\lambda_i}} + g(s)\right)ds = -\infty, \\
        & \limsup_{t \to \infty}t^{\beta-\alpha-n+1}\int_T^t (t-s)^{\alpha-1}\left(-\gamma\sum_{i=1}^k |p(s)|^{\frac{\lambda_i}{\lambda_i-1}}|q_i(s)|^{\frac{1}{1-\lambda_i}} + g(s)\right)ds = + \infty,
    \end{align*}
    where 
    \begin{align*}
        &\gamma: = \left( \frac{1}{k \lambda_1}\right)^{\frac{\lambda_k}{\lambda_k-1}},\,\,\, \text{if}\; (\textup{C})_1\; \text{holds}, \\
        &\gamma:= \max \left\{ \max_{i \in \{1,\ldots,l\}} (1-
        \lambda_i)\left( \frac{l\lambda_i}{\sum_{j=l+1}^k \lambda_j -1} \right)^{\frac{\lambda_i}{1-\lambda_i}}, \lambda_k-1 \right\},\,\,\, \text{if}\;(\textup{C})_2\;  \text{holds}.
    \end{align*}
\end{itemize}
\begin{theorem}\label{dl3.5}
    Assume that the conditions (C), (D), (E), and (F) hold. Then an arbitrary solution of \eqref{eq5} (if it exists globally) is oscillatory. 
\end{theorem}
\begin{proof}
    Suppose to the contrary that there exists a non-oscillatory solution $x$ of system \eqref{eq5}--\eqref{dkd2}. Without loss of generality, we can assume that $x(t) > 0$ for $t > t_1.$ Due to the assumptions (D) and (E), we can find $t_2 > \max\{t_1,T\}$ such that $p(t) \neq 0$ and $\mathrm{sgn}(\lambda_i-1)q_i(t) > 0,\; i = 1,2, \ldots,k$, for $t \geq t_2$. Using the same arguments as in the proof of Theorem \ref{dl1}, for all $t > t_2$, we obtain
    \begin{align*}
        x(t)= &\sum_{i=0}^{n-1}\frac{x^{(i)}(0)}{i!}t^i + \sum_{i=0}^{m-1}\frac{ax^{(i)}(0)}{\Gamma(\alpha-\beta+1+i)}t^{\alpha-\beta+i} -\frac{a}{\Gamma(\alpha-\beta)}\int_0^t(t-s)^{\alpha-\beta-1}x(s)ds\\
        &-\frac{1}{\Gamma(\alpha)}\int_0^t(t-s)^{\alpha-1}p(s)x(s)ds -  \frac{1}{\Gamma(\alpha)} \sum_{i=1}^k \int_0^t (t-s)^{\alpha-1}q_i(s)\mathrm{sgn}\left(x(s)\right)|x(s)|^{\lambda_i}ds \\
        &+\frac{1}{\Gamma(\alpha)} \int_0^t (t-s)^{\alpha-1}g(s)ds\\
        &=\sum_{i=0}^{n-1}\frac{x^{(i)}(0)}{i!}t^i + \sum_{i=0}^{m-1}\frac{ax^{(i)}(0)}{\Gamma(\alpha-\beta+1+i)}t^{\alpha-\beta+i} -\frac{a}{\Gamma(\alpha-\beta)}\int_0^{t_2}(t-s)^{\alpha-\beta-1}x(s)ds\\
        &-\frac{1}{\Gamma(\alpha)}\int_0^{t_2}(t-s)^{\alpha-1}p(s)x(s)ds -  \frac{1}{\Gamma(\alpha)} \sum_{i=1}^k \int_0^{t_2} (t-s)^{\alpha-1}q_i(s)\mathrm{sgn}\left(x(s)\right)|x(s)|^{\lambda_i}ds\\
        &+\frac{1}{\Gamma(\alpha)} \int_0^T (t-s)^{\alpha-1}g(s)ds-\frac{a}{\Gamma(\alpha-\beta)}\int_{t_2}^t(t-s)^{\alpha-\beta-1}x(s)ds\\
        &-\frac{1}{\Gamma(\alpha)}\int_{t_2}^t(t-s)^{\alpha-1}p(s)x(s)ds -  \frac{1}{\Gamma(\alpha)} \sum_{i=1}^k \int_{t_2}^t (t-s)^{\alpha-1}q_i(s)\mathrm{sgn}\left(x(s)\right)|x(s)|^{\lambda_i}ds \\
        &+\frac{1}{\Gamma(\alpha)} \int_T^t (t-s)^{\alpha-1}g(s)ds,
    \end{align*} 
    where $m= \lceil  \beta \rceil$. Notice that $(t-s)^{\alpha-\beta-1} \leq (t-t_2)^{\alpha-\beta-1}$ for all $s \in [0, t_2]$, $t>t_2$ if $\alpha-\beta-1 \leq 0$ and $(t-s)^{\alpha-\beta-1} \leq t^{\alpha-\beta-1}$ for all $s \in [0, t_2]$, $t>t_2$ if $\alpha-\beta-1 > 0$. Thus, $(t-s)^{\alpha-\beta-1} \leq (t-t_2)^{\alpha-\beta-1} +  t^{\alpha-\beta-1}$ for all $s \in [0, t_2]$, which implies
    \begin{align}\label{3.2.17}
         -\frac{a}{\Gamma(\alpha-\beta)}\int_0^{t_2}(t-s)^{\alpha-\beta-1}x(s)ds &\leq \frac{a}{\Gamma(\alpha-\beta)}\left( (t-t_2)^{\alpha-\beta-1} +  t^{\alpha-\beta-1}\right)\int_0^{t_2}|x(s)|ds \nonumber\\
         & = M_1\left( (t-t_2)^{\alpha-\beta-1} +  t^{\alpha-\beta-1}\right),\;t>t_2.
    \end{align}
    Furthermore, it is not difficult to check 
    \begin{align}\label{3.2.18}
        -\frac{1}{\Gamma(\alpha)}\int_0^{t_2}(t-s)^{\alpha-1}p(s)x(s)ds \leq M_2\left((t-t_2)^{\alpha-1} + t^{\alpha-1}\right),
    \end{align}
    \begin{align}\label{3.2.19}
        \frac{1}{\Gamma(\alpha)} \sum_{i=1}^k \int_0^{t_2} (t-s)^{\alpha-1}q_i(s)\mathrm{sgn}\left(x(s)\right)|x(s)|^{\lambda_i}ds \leq M_3\left((t-t_2)^{\alpha-1} + t^{\alpha-1}\right),
    \end{align}
    and
    \begin{align}\label{3.2.20}
        \frac{1}{\Gamma(\alpha)} \int_0^T (t-s)^{\alpha-1}g(s)ds \leq M_4\left((t-t_2)^{\alpha-1} + t^{\alpha-1}\right)
    \end{align}
    for all $t>t_2$. Take $b: = \max\{|x_0|,|x_1|,\ldots,|x_{n-1}|\}$. For $t \geq t_2$, it follows from \eqref{3.2.17}, \eqref{3.2.18}, 
  \eqref{3.2.19} and \eqref{3.2.20} that
    \begin{align}\label{pt16}
        x(t) \leq&\; \sum_{k=0}^{n-1}\frac{b}{k!}t^k + \sum_{k=0}^{m-1}\frac{ab}{\Gamma(\alpha-\beta+1+k)}t^{\alpha-\beta+k} + M_1\left( (t-t_2)^{\alpha-\beta-1} +  t^{\alpha-\beta-1}\right)\nonumber \\
        & + (M_2+M_3+M_4)\left((t-t_2)^{\alpha-1} + t^{\alpha-1}\right)+\frac{1}{\Gamma(\alpha)} \int_{T}^t (t-s)^{\alpha-1}g(s)ds \nonumber\\
        &+ \frac{1}{\Gamma(\alpha)}\int_{t_2}^t (t-s)^{\alpha-1}\left(|p(s)|x(s) -\sum_{i=1}^m q_i(s)x(s)^{\lambda_i}\right)ds.
    \end{align}
    Consider the case when $(\textup{C})_1$ is true, that is, $1 < \lambda_1 < \lambda_2 < \cdots < \lambda_k$. We get
    $q_i(t) > 0,\;i=1,2,\ldots,k$ and $t \ge t_2$. Then,
    \begin{align}\label{m1}
        |p(s)|x(s) - \sum_{i=1}^k q_i(s)x(s)^{\lambda_i} = \sum_{i=1}^k \left(\frac{1}{k}|p(s)|x(s) - q_i(s)x(s)^{\lambda_i}\right), \; s \geq t_2.
    \end{align}
    For each $i=1,2,\ldots,k$, let 
    \begin{align*}
        X_i(s): = q_i(s)^{\frac{1}{\lambda_i}}x(s), Y_i(s): = \left(\frac{1}{m\lambda_i}|p(s)|q_i(s)^{\frac{-1}{\lambda_i}}\right)^{\frac{1}{\lambda_i-1}}, \; s \geq t_2.
    \end{align*}
    Due to Lemma \ref{bdt}(i), we see 
    \begin{align*}
        \lambda_i q_i(s)^{\frac{1}{\lambda_i}}x(s) \frac{1}{k\lambda_i}|p(s)|q_i(s)^{\frac{-1}{\lambda_i}} - q_i(s)x(s)^{\lambda_i} \leq (\lambda_i-1)\left(\frac{1}{k\lambda_i}|p(s)|q_i(s)^{\frac{-1}{\lambda_i}}\right)^{\frac{\lambda_i}{\lambda_i-1}},
    \end{align*}
    this means that
    \begin{align}\label{m2}
        \frac{1}{m}|p(s)|x(s) - q_i(s)x(s)^{\lambda_i}&\leq (\lambda_i-1)\left(\frac{1}{m\lambda_i}|p(s)|q_i(s)^{\frac{-1}{\lambda_i}}\right)^{\frac{\lambda_i}{\lambda_i-1}} \nonumber \\
        &\leq \gamma |p(s)|^{\frac{\lambda_i}{\lambda_i-1}}|q_i(s)|^{\frac{1}{1-\lambda_i}}.
    \end{align}
    From \eqref{m1} and \eqref{m2}, it leads to 
    \begin{align}\label{ss1}
        |p(s)|x(s) - \sum_{i=1}^k q_i(s)x(s)^{\lambda_i} \leq \gamma \sum_{i=1}^k |p(s)|^{\frac{\lambda_i}{\lambda_i-1}}|q_i(s)|^{\frac{1}{1-\lambda_i}}, \quad s \geq t_2.
    \end{align}
    If the assumption $(\textup{C})_2$ holds, that is, there is an index $l$ satisfying $ \lambda_1  < \cdots < \lambda_l < 1 < \lambda_{l+1} < \cdots < \lambda_k$. It implies that 
    $q_i(t) < 0 ,i=1,\ldots,l$ and $q_i(t) > 0 ,i=l+1,\ldots,k$ for $t \geq t_2$. In this case, we write
    \begin{align}\label{pt17}
        |p(s)|x(s) - \sum_{i=1}^k &q_i(s)x(s)^{\lambda_i} = |p(s)|x(s) + \sum_{i=1}^l|q_i(s)|x(s)^{\lambda_i} - \sum_{i=l+1}^k q_i(s)x(s)^{\lambda_i}\nonumber\\
        &\hspace{-2cm}= \sum_{i=1}^l\left(|q_i(s)|x(s)^{\lambda_i} - A|p(s)|x(s)\right)+ \sum_{i=l+1}^k\left(\lambda_i|p(s)|x(s) -q_i(s)x(s)^{\lambda_i}\right), 
    \end{align}
    where $A = \frac{\sum_{j=l+1}^k \lambda_j -1}{l} >0$. For each $i=1,\ldots,l$, define
    \begin{align*}
        X_i(s): = |q_i(s)|^{\frac{1}{\lambda_i}}x(s), Y_i(s): = \left(\frac{A}{\lambda_i}|p(s)||q_i(s)|^{\frac{-1}{\lambda_i}}\right)^{\frac{1}{\lambda_i-1}}, \; s \geq t_2.
    \end{align*}
    By Lemma \ref{bdt}(ii),
    \begin{align*}
        \lambda_i |q_i(s)|^{\frac{1}{\lambda_i}}x(s)\frac{A}{\lambda_i}|p(s)|\,|q_i(s)|^{\frac{-1}{\lambda_i}} - |q_i(s)|x(s)^{\lambda_i} \geq (\lambda_i-1)\left(\frac{A}{\lambda_i}|p(s)|\,|q_i(s)|^{\frac{-1}{\lambda_i}}\right)^{\frac{\lambda_i}{\lambda_i-1}}.
    \end{align*}
    Thus,
    \begin{align}\label{pt18}
        |q_i(s)|x(s)^{\lambda_i} - A|p(s)|x(s) &\leq (1-\lambda_i)\left(\frac{A}{\lambda_i}|p(s)||q_i(s)|^{\frac{-1}{\lambda_i}}\right)^{\frac{\lambda_i}{\lambda_i-1}} \nonumber\\
        & \leq \gamma |p(s)|^{\frac{\lambda_i}{\lambda_i-1}}|q_i(s)|^{\frac{1}{1-\lambda_i}},\; i=1,\ldots,l, s\geq t_2.
    \end{align}
    For each $i=l+1,\ldots,m$, let
    \begin{align*}
        X_i(s):= q_i(s)^{\frac{1}{\lambda_i}}x(s), Y_i(s):= \left(|p(s)|q_i(s)^{\frac{-1}{\lambda_i}}\right)^{\frac{1}{\lambda_i-1}}, \quad s \geq t_2. 
    \end{align*}
     From Lemma \ref{bdt}(i), 
    \begin{align*}
        \lambda_i q_i(s)^{\frac{1}{\lambda_i}}x(s)|p(s)|q_i(s)^{\frac{-1}{\lambda_i}} - q_i(s)x(s)^{\lambda_i} \leq (\lambda_i-1)\left(|p(s)|q_i(s)^{\frac{-1}{\lambda_i}}\right)^{\frac{\lambda_i}{\lambda_i-1}},
    \end{align*}
    which gives
    \begin{align}\label{pt19}
        \lambda_i |p(s)|x(s) - q_i(s)x(s)^{\lambda_i} &\leq (\lambda_i-1)\left(|p(s)|q_i(s)^{\frac{-1}{\lambda_i}}\right)^{\frac{\lambda_i}{\lambda_i-1}} \nonumber \\
        & \leq \gamma |p(s)|^{\frac{\lambda_i}{\lambda_i-1}}|q_i(s)|^{\frac{1}{1-\lambda_i}},\; i=l+1,\ldots,m,s\geq t_2.
    \end{align}
    Combining \eqref{pt17}, \eqref{pt18} and \eqref{pt19} yields
    \begin{align}\label{pt20}
        |p(s)|x(s) - \sum_{i=1}^k q_i(s)x(s)^{\lambda_i} \leq \gamma\sum_{i=1}^k  |p(s)|^{\frac{\lambda_i}{\lambda_i-1}}|q_i(s)|^{\frac{1}{1-\lambda_i}},\; s \geq t_2.
    \end{align}
    In short, if the condition (C) is true, from \eqref{ss1} and \eqref{pt20}, the following estimate is derived
    \begin{align}\label{m4}
        |p(s)|x(s) - \sum_{i=1}^k q_i(s)x(s)^{\lambda_i} \leq \gamma\sum_{i=1}^k  |p(s)|^{\frac{\lambda_i}{\lambda_i-1}}|q_i(s)|^{\frac{1}{1-\lambda_i}},\; s \geq t_2,
    \end{align}
    which together with \eqref{pt16} implies
    \begin{align*}
        t^{\beta-\alpha-n+1}x(t) &\leq \sum_{k=0}^{n-1}\frac{b}{k!}t^{k + \beta-\alpha-n+1} + \sum_{k=0}^{m-1}\frac{ab}{\Gamma(\alpha-\beta+1+k)}t^{ k - n+1 } \\
        &+ M_1\left((t-t_2)^{\alpha-\beta-1}t^{\beta-\alpha-n+1} + t^{-n}\right) \\
        &+ (M_2 + M_3 + M_4)\left((t-t_2)^{\alpha-1}t^{\beta-\alpha-n+1} + t^{\beta-n}\right)\nonumber \\
        & + \frac{t^{\beta-\alpha-n+1}}{\Gamma(\alpha)}\int_{T}^t (t-s)^{\alpha-1}\left(\gamma\sum_{i=1}^m  |p(s)|^{\frac{\lambda_i}{\lambda_i-1}}|q_i(s)|^{\frac{1}{1-\lambda_i}} + g(s)\right)ds,\;t>t_2.
    \end{align*}
    Therefore, there exist $c >0 $ and $t_3 > t_2$ such that
    \begin{align*}
        t^{\beta-\alpha-n+1}x(t) \leq c +  \frac{t^{\beta-\alpha-n+1}}{\Gamma(\alpha)}\int_{T}^t (t-s)^{\alpha-1}\left(\gamma\sum_{i=1}^m  |p(s)|^{\frac{\lambda_i}{\lambda_i-1}}|q_i(s)|^{\frac{1}{1-\lambda_i}} + g(s)\right)ds, \; t > t_3.
    \end{align*}
    Under the assumption (F), we conclude
    \begin{align*}
        \liminf_{t \to \infty} t^{\beta-\alpha-n+1}x(t) = - \infty,
    \end{align*}
    contrary to the counterfactual hypothesis that $x(t) > 0$ for all $t \geq t_2$. The proof is complete.
\end{proof}
\begin{example}\label{vd1}
    Consider the initial value problem \eqref{eq5}--\eqref{dkd2}, here $\alpha = 1/2$,  $\beta = 1/3$, $a =2$, $p(t) = -t$, $\lambda_1 = 1/2$, $\lambda_2 =2$, $q_1(t) = -t^{2}$, $q_2(t) =t^{3}$, $g(t) = ^C D^{1/2}_{0^+} (t^6 \sin t)$ for $t\geq 0$. This equation is rewritten as
    \begin{align}\label{bsvd1}
        ^CD^{1/2}_{0^+}x(t) + 2\,\,^CD^{1/3}_{0^+}x(t)-tx(t) -t^2\mathrm{sgn}\left(x(t)\right)|x(t)|^{1/2} + t^3 \mathrm{sgn}\left(x(t)\right)|x(t)|^2=^C D^{1/2}_{0^+} (t^6 \sin t).
    \end{align} 
    Then it is easy to check that $\lambda_1, \lambda_2$ satisfy $\textup{(C)}_2$, $p$ satisfies (D) and $q_1, q_2$ satisfy (E) for $T_1 = T_2 =1$, respectively. Next, we see $\gamma = 1$ and
    \begin{align*}
        \gamma\sum_{i=1}^m  |p(s)|^{\frac{\lambda_i}{\lambda_i-1}}|q_i(s)|^{\frac{1}{1-\lambda_i}} =  t^3 + t^5, \; t\geq 0.
    \end{align*}
    Notice that
    \begin{align*}
    g(t) = \frac{1}{\Gamma(1/2)}\int_0^t \frac{s^5 \sin s + s^6 \cos s}{(t-s)^{1/2}}ds ,\; t \geq 0,
    \end{align*}
   which implies
    \begin{align*}
        |g(t)| \leq \frac{1}{\Gamma(1/2)}\int_0^t \frac{s^5 + s^6}{(t-s)^{1/2}}ds = \frac{5!}{\Gamma(13/2)} t^{11/2} + \frac{6!}{\Gamma(15/2)}t^{13/2},\;t\geq 0.
    \end{align*}
    Hence, for $t > 1$, we obtain 
    \begin{align*}
        \int_1^t &(t-s)^{-1/2}\left(\gamma\sum_{i=1}^m  |p(s)|^{\frac{\lambda_i}{\lambda_i-1}}|q_i(s)|^{\frac{1}{1-\lambda_i}} + g(s)\right)ds < \Gamma(1/2) I^{1/2}_{0^+}\left(t^3 +t^5 +^C D^{1/2}_{0^+} (t^6 \sin t) \right) \\
        &+ 2M(t-1)^{1/2} \\
        &\hspace{2cm}<\frac{3!}{\Gamma(9/2)}t^{7/2} + \frac{6!}{\Gamma(13/2)}t^{11/2} + \Gamma(1/2) t^6 \sin t +  2Mt^{1/2},
    \end{align*}
    where $M = \max_{t \in [0,1]} (t^3 +t^5 + |g(t)|) = 2 +  \frac{5!}{\Gamma(13/2)} + \frac{6!}{\Gamma(15/2)}$.
    From this,
    \begin{align*}
        \liminf_{t\to \infty} t^{-1/2}\int_1^t &(t-s)^{-1/2}\left(\gamma\sum_{i=1}^m  |p(s)|^{\frac{\lambda_i}{\lambda_i-1}}|q_i(s)|^{\frac{1}{1-\lambda_i}} + g(s)\right)ds \\
        &\leq \liminf_{t \to \infty}(\frac{3!}{\Gamma(9/2)}t^{3} + \frac{6!}{\Gamma(13/2)}t^{5} + \Gamma(1/2) t^{11/2} \sin t +  2M) = -\infty.
    \end{align*}
 By the same arguments, we also see
    \begin{align*}
        \limsup_{t\to \infty} t^{-1/2}\int_1^t &(t-s)^{-1/2}\left(-\gamma\sum_{i=1}^m  |p(s)|^{\frac{\lambda_i}{\lambda_i-1}}|q_i(s)|^{\frac{1}{1-\lambda_i}} + g(s)\right)ds \\
        &\geq \limsup_{t \to \infty}(-\frac{3!}{\Gamma(9/2)}t^{3} - \frac{6!}{\Gamma(13/2)}t^{5} + \Gamma(1/2) t^{11/2} \sin t + 2M) = +\infty,
    \end{align*}
    and thus the condition (F) is verified. From Theorem \ref{dl3.5}, we conclude that all solutions  (if they exist globally) are oscillatory. Figure 2 depicts the orbit of the solution with the initial condition $x(0)=1$ on the interval $[0,80]$. 
    \begin{figure}
		\begin{center}
			\includegraphics[scale=.7]{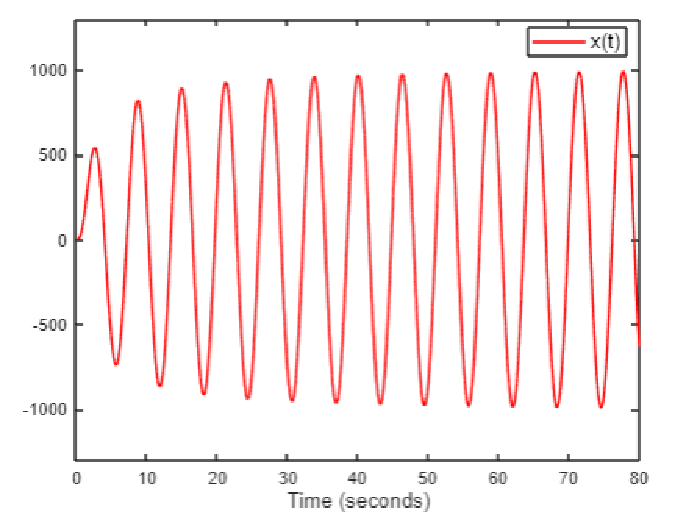}
		\end{center}
		\begin{center}
			\caption{The orbit of the solution to equation \eqref{bsvd1} with the initial condition $x_0 = 1$ on the interval $[0,80]$.}
		\end{center}
\end{figure}
\end{example}
\section{On the asymptotic behavior of oscillatory solutions of fractional differential equations}
This section is devoted to the asymptotic behavior of oscillatory solutions of some classes of fractional differential equations. We first present some comparison results for equations with two fractional derivatives to do this. Then, we consider linear equations. Finally, we deal with the nonlinear case by combining Theorem \ref{dl1}, the approach as in \cite{Thinh}, and the proposed comparison principles.
\subsection{Comparison results for multi-order fractional differential equations}
For given a parameter $T>0$, consider the equation
\begin{align}
   ^C D^{\alpha}_{0^+}x(t) +a\,  ^C D^{\beta}_{0^+}x(t)& = f(t,x(t)),\;t\in (0,T],\label{ss1_T}\\
   x(0)=x_0\in\mathbb R,\label{ss2_T}
\end{align}
where $0 < \beta < \alpha \leq 1$, $a$ is a positive constant, $f:[0,T]\times \mathbb R\rightarrow \mathbb R$ is a continuous function and Lipschitz continuous with respect to the second variable. We will present some comparison results of the solutions of the system \eqref{ss1_T}--\eqref{ss2_T}.
\begin{lemma}\label{bdss1}
Let $x(\cdot)$ be the unique solution of the initial value problem \eqref{ss1_T}--\eqref{ss2_T}. Suppose that $y : [0, T] \to \mathbb R$ is continuously Caputo $\alpha$-differentiable satisfying
    \begin{align*}
        ^C D^{\alpha}_{0^+}y(t) +a\,  ^C D^{\beta}_{0^+}y(t) &\leq f(t,y(t)) , \; t \in (0,T], \\
        y(0) = y_0 < x_0, 
    \end{align*}
     then $y(t) < x(t)$ for all $t \in [0,T]$.
\end{lemma}
    \begin{proof}
    Suppose by contradiction that there is a $t\in (0,T]$ such that $y(t) = x(t)$.
       Due to $y(0) = y_0 < x_0 = x(0)$ and $x,y \in C[0,T]$, there exists a $t_1 \in (0,T]$ such that
       \begin{align}
           y(t)& < x(t),\;\forall t\in [0,t_1),\label{ss3_T}\\
           y(t_1)&=x(t_1)\label{ss4_T}
       \end{align}
       Define
        \begin{align*}
            z(t): = x(t) -y(t),\; t\in [0, T].
        \end{align*}
        Then $z(t) > 0$ for all $t \in [0, t_1)$ and $z(t_1) = 0$. Notice that for $0<\alpha<1$,  by virtue of Lemma \ref{bdv}, we have
        \begin{align}\label{3.30}
            ^C D^\alpha_{0^+}z(t_1) &= \frac{1}{\Gamma(1-\alpha)}\left(\frac{z(t_1)-z(0)}{t_1^{\alpha}}\right) + \alpha \int_0^{t_1}\frac{z(t_1)-z(s)}{(t_1-s)^{\alpha+1}}ds\nonumber\\
            &=\frac{1}{\Gamma(1-\alpha)}\left(\frac{-z(0)}{t_1^{\alpha}}\right) + \alpha \int_0^{t_1}\frac{-z(s)}{(t_1-s)^{\alpha+1}}ds < 0.
        \end{align}
        In the case $\alpha =1$, then 
        \begin{align}
            z'(t_1)& = \lim_{h \to 0^-}\frac{z(t_1+h) -z(t_1)}{h}\notag\\
            &=\lim_{h \to 0^-}\frac{z(t_1+h)}{h}\notag\\
            &\leq 0\label{3.31}.
        \end{align}
        Thus, from \eqref{3.30} and \eqref{3.31}, we claim
        \begin{align}\label{3.32}
            ^C D^\alpha_{0^+}z(t_1)  \leq 0.
        \end{align}
        Furthermore, due to $z(\cdot)$ is continuously Caputo $\alpha$-differentiable on $[0, T]$ and $0 < \beta < \alpha \leq 1$, it follows from \cite[Theorem 3.8 (i)]{Cong} that this function is also continuously Caputo $\beta$-differentiable and thus, by the same arguments as shown above, $^C D^\beta_{0^+}z(t_1) < 0$. This together with \eqref{3.32} implies
        \begin{align*}
            ^C D^\alpha_{0^+}z(t_1) + a\,^C D^\beta_{0^+}z(t_1) < 0,
        \end{align*}
        that is,
        \begin{equation}\label{ss5_T}
            ^C D^\alpha_{0^+}x(t_1) + a\,^C D^\beta_{0^+}x(t_1) < ^C D^\alpha_{0^+}y(t_1) + a\,^C D^\beta_{0^+}y(t_1).
        \end{equation}
        However, 
        \begin{equation*}
            ^C D^\alpha_{0^+}x(t_1) + a\,^C D^\beta_{0^+}x(t_1) =f(t_1,x(t_1)) = f(t_1,y(t_1)) \geq ^C D^\alpha_{0^+}y(t_1) + a\,^C D^\beta_{0^+}y(t_1),
        \end{equation*}
        contrary to \eqref{ss5_T}. Therefore, $y(t) < x(t)$ for all $t \in [0,T]$.
    \end{proof}
    \begin{theorem}\label{ssnh}
        Let $x(\cdot)$ be the unique solution of the initial value problem \eqref{ss1_T}--\eqref{ss2_T}. Assume that
    $y : [0, T] \to \mathbb R$ is continuously Caputo $\alpha$-differentiable satisfying
    \begin{align*}
        ^C D^{\alpha}_{0^+}y(t) +a\,  ^C D^{\beta}_{0^+}y(t) &\leq f(t,y(t)) , \; t \in (0,T], \\
        y(0) = y_0 \leq  x_0, 
    \end{align*}
     then $y(t) \leq x(t)$ for all $t \in [0,T]$.
    \end{theorem}
    \begin{proof}
    For each $m \in \mathbb N$, by Theorem \ref{dlttdn}, the following initial value problem 
    \begin{align}
        ^C D^{\alpha}_{0^+}x(t) +a\,  ^C D^{\beta}_{0^+}x(t) &= f(t,x(t)), \; t \in (0,T],\label{ss6_T} \\
        x(0)& = x_0 + \frac{1}{m},\label{ss7_T}
    \end{align}
    has a unique solution $x_m :[0,T] \to \mathbb R$. Because $y(0) =y_0 \leq x_0 < x_0 + \frac{1}{m} =x_m(0)$ for all $m \in \mathbb N$, according to Lemma \ref{bdss1}, we see
    \begin{align}\label{3.33}
        y(t) < x_m(t)\leq x_1(t),\;\forall t \in [0,T].
    \end{align}
    In addition, for $n,m\in \mathbb N$ and $n>m$, then
    \begin{align*}
        x_n(t)<x_m(t),\;\forall t\in [0,T].
    \end{align*}
    Thus, the sequence $\{x_m(\cdot)\}_{m\in\mathbb N}$ is uniformly bounded on $[0,T]$ and for each $t\in[0,T]$, the sequence $\{x_m(t)\}_{m\in\mathbb N}$ is strictly decreasing. We next show that this sequence is equicontinuous. Let $C_1 > 0$ such that $|x_m(t)| \leq C_1$ for all $t \in [0,T]$ and $m \in \mathbb N$ and take $C_2:= \max_{[0,T]\times[-C_1,C_1]}|f(t,x)|$. For $0 \leq t_1 < t_2 \leq T$ and $m \in \mathbb N$, based on \cite[Lemma 2.1]{Thinh}, we obtain the following estimates 
    \begin{align}
        |x_m(t_2)-x_m(t_1)| =& \;\left| \frac{a}{\Gamma(\alpha-\beta)}\left(\int_0^{t_1}(t_1-s)^{\alpha-\beta-1}x_m(s)ds -\int_0^{t_2}(t_2-s)^{\alpha-\beta-1}x_m(s)ds \right) \right.\notag \\
        &+ \left. \frac{1}{\Gamma(\alpha)}\left(\int_0^{t_2} (t_2-s)^{\alpha-1}f(s,x_m(s)ds -\int_0^{t_1} (t_1-s)^{\alpha-1}f(s,x_m(s)ds \right) \right|\notag \\
        \leq &\;\frac{a}{\Gamma(\alpha-\beta)}\int_0^{t_1}\left((t_1-s)^{\alpha-\beta-1} -(t_2-s)^{\alpha-\beta-1}\right)|x_m(s)|ds\notag \\
        &+\frac{a}{\Gamma(\alpha-\beta)}\int_{t_1}^{t_2}(t_2-s)^{\alpha-\beta-1}|x_m(s)|ds \notag\\
        &+\frac{1}{\Gamma(\alpha)}\int_0^{t_1}\left((t_1-s)^{\alpha-1} -(t_2-s)^{\alpha-1}\right)|f(s,x_m(s)|ds\notag\\
        &+\frac{1}{\Gamma(\alpha)}\int_{t_1}^{t_2}(t_2-s)^{\alpha-1}|f(s,x_m(s)|ds\notag\\
        \leq &\;\frac{a C_{1}}{\Gamma(\alpha-\beta+1)}(t_1^{\alpha-\beta} + (t_2 -t_1)^{\alpha-\beta} -t_2^{\alpha-\beta}) + \frac{a C_{1}}{\Gamma(\alpha-\beta+1)}(t_2-t_1)^{\alpha-\beta}\notag\\
        &+\frac{ C_{2}}{\Gamma(\alpha+1)}(t_1^{\alpha} + (t_2 -t_1)^{\alpha} -t_2^{\alpha}) + \frac{C_{2}}{\Gamma(\alpha+1)}(t_2-t_1)^{\alpha}\notag\\
        \leq&\; \frac{2a C_{1}}{\Gamma(\alpha-\beta+1)}(t_2-t_1)^{\alpha-\beta} + \frac{ 2C_{2}}{\Gamma(\alpha+1)}(t_2-t_1)^{\alpha}.\label{compare_1}
    \end{align}
   Let $\epsilon > 0$ be arbitrarily small. Choosing
    \begin{align*}
       0< \delta < \min \left\{\left(\frac{\epsilon}{2}\frac{\Gamma(\alpha-\beta+1)}{2aC_1}\right)^{\frac{1}{\alpha-\beta}}, \left(\frac{\epsilon}{2}\frac{\Gamma(\alpha+1)}{2C_2}\right)^{\frac{1}{\alpha}}\right\},
    \end{align*}
    from \eqref{compare_1}, we get
    \begin{align*}
        |x_m(t_2)-x_m(t_1)| < \frac{\epsilon}{2} + \frac{\epsilon}{2} = \epsilon
    \end{align*}
    for all $t_1,t_2\in [0,T]$ with $|t_1-t_2|<\delta$.  This means that $\{x_m(\cdot)\}_{m \in \mathbb N}$ is equicontinuous on $[0, T]$. By Arzel$\grave{a}$--Ascoli theorem, there is a subsequence $\{x_{m_k}(\cdot)\}_{k\in \mathbb N}$ of $\{x_m(\cdot)\}_{m\in\mathbb N}$ which uniformly converges to a function $x^*(\cdot)$ on $[0, T]$. In particular, $y(t)\leq x^*(t)$ for all $t\in [0,T]$. Notice that, for each $k\in \mathbb N$ and $t\in [0,T]$, we have (by \cite[Lemma 2.1]{Thinh})  
    \begin{align*}
        x_{m_k}(t) = x_0 + \frac{1}{m_k} &+ \frac{a(x_0+\frac{1}{m_k})}{\Gamma(1+\alpha-\beta)}t^{\alpha-\beta}-\frac{a}{\Gamma(\alpha-\beta)}\int_0^t(t-s)^{\alpha-\beta-1}x_{m_k}(s)ds\\
        &+ \frac{1}{\Gamma(\alpha)}\int_0^t (t-s)^{\alpha-1}f(s,x_{m_k}(s))ds.
    \end{align*}
    Letting $k \to \infty$, then 
\begin{align}\label{ss8_T}
x^*(t)&= x_0 + \frac{ax_0}{\Gamma(1+\alpha-\beta)}t^{\alpha-\beta}-\frac{a}{\Gamma(\alpha-\beta)}\int_0^t(t-s)^{\alpha-\beta-1}x^*(s)ds\notag\\
&\hspace{2cm}+\frac{1}{\Gamma(\alpha)}\int_0^t (t-s)^{\alpha-1}f(s,x^*(s))ds,\;t\in [0,T].
\end{align}
    On the other hand, due to Theorem \ref{dlttdn}, the original system \eqref{ss1_T}--\eqref{ss2_T} has a unique solution $x(\cdot)$, which is also written in the form \eqref{ss8_T} (by \cite[Lemma 2.1]{Thinh}). This implies $x^* \equiv x$ on $[0, T]$ and thus $y(t)\leq x(t)$ for all $t\in [0,T]$.
    \end{proof}
     \begin{corollary}\label{sslh}
         Let $x(\cdot)$ be the unique solution of the initial value problem \eqref{ss1_T}--\eqref{ss2_T}. Assume that
    $y : [0, T] \to \mathbb R$ is continuously  Caputo $\alpha$-differentiable satisfying
    \begin{align*}
        ^C D^{\alpha}_{0^+}y(t) +a\,  ^C D^{\beta}_{0^+}y(t) &\geq f(t,y(t)) , \; t \in (0,T], \\
        y(0) = y_0 \geq  x_0, 
    \end{align*}
     then $y(t) \geq x(t)$ for all $t \in [0,T]$.
     \end{corollary}
     \begin{proof}
         Put $u(t) = -x(t)$, $v(t) =-y(t)$, $t \in [0, T]$ and define $g(t,x) = -f(t,-x)$, $t \in [0,T], x\in \mathbb R$. Then, the function $g(\cdot,\cdot)$ is continuous and is Lipschitz continuous to the second variable. It is easy to check that
         \begin{align*}
             ^C D^{\alpha}_{0^+}u(t) +a\,  ^C D^{\beta}_{0^+}u(t) &= -^C D^{\alpha}_{0^+}x(t) -a\,  ^C D^{\beta}_{0^+}x(t) \\
             &= -f(t,x(t)) = g(t,-x(t))=g(t,u(t)),\;t\in (0,T],\\
             u(0)&= -x_0. 
         \end{align*}
        Similarly,  $v(\cdot)$ is continuously Caputo $\alpha$-differentiable and satisfies 
        \begin{align*}
        ^C D^{\alpha}_{0^+}v(t) +a\,  ^C D^{\beta}_{0^+}v(t) &\leq g(t,v(t)) ,\; t \in (0,T], \\
        y(0) = -y_0 \leq -x_0.
        \end{align*}
        On account of Theorem \ref{ssnh} above, we have $v(t) \leq u(t),\; t\in [0,T]$, that is, $y(t) \geq x(t),\; t \in [0,T]$.
     \end{proof}
    \subsection{Asymptotic behavior of oscillatory solutions of fractional-order linear equations}
    Consider the linear equation
    \begin{align}\label{bt3}
        ^C D^\alpha_{0^+}x(t) + a \, ^C D^\beta_{0^+}x(t) = -bx(t) +  g(t), \quad t> 0
    \end{align}
    with the initial value
    \begin{align}\label{dkbt3}
        x(0) = x_0,
    \end{align}
    where $0 < \beta < \alpha \leq 1$, $a,b$ are positive real constants, $g : [0, \infty) \to \mathbb R$ is a continuous function.
    \begin{theorem}\label{dl4.4}
        Suppose that $g(\cdot)$ satisfies the following assumption.
    \end{theorem}
    \begin{itemize}
        \item [(B)''] For $t\in [0,\infty)$, $g(t): = ^C D^\alpha_{0^+} (t+1)^{\sigma}h(t)$, here $\alpha-\beta < \sigma < \alpha \leq 1$ and $h(\cdot)$ is  continuously differentiable and bounded on $[0,\infty)$ such that $\limsup_{t \to \infty}h(t) > 0$, $\liminf_{t \to \infty}h(t) < 0$. In addition, assume that $h'(t) = O(t^{\eta})$ as $t \to \infty$, here $-1< \eta < \alpha-\sigma-1$.
    \end{itemize}
    Then, an arbitrary nontrivial solution $x(\cdot)$ of the initial value problem \eqref{bt3}--\eqref{dkbt3} is oscillatory. Furthermore, its asymptotic behavior has the form $O(t^{\max\{-\beta, \eta + \sigma-\alpha+1\}})$ at infinity. 
    \begin{proof}
    It is easy to check that for a function $g(\cdot)$ satisfying the condition $(\textup{B})$'', it also satisfies the assumption (B) in Theorem  \ref{dl1}. Thus, all non-trivial solutions of the system \eqref{bt3}--\eqref{dkbt3} are oscillatory. Next, we show that $g(t) = O(t^{\eta + \sigma-\alpha+1})$ as $t \to \infty$. Indeed, for $t>0$, we have
        \begin{align}\label{4.3.38}
            g(t) &= ^C D^\alpha_{0^+} (t+1)^\sigma h(t) =\frac{1}{\Gamma(1-\alpha)}\int_0^t \frac{(s+1)^{\sigma-1}h(s) + (s+1)^{\sigma}h'(s)}{(t-s)^\alpha}ds.  
        \end{align}
        Notice that $h(\cdot)$ is bounded on $[0,\infty)$, there is a $M_1>0$ such that
        \begin{align}\label{4.3.39}
            \left|\frac{1}{\Gamma(1-\alpha)}\int_0^t \frac{(s+1)^{\sigma-1}h(s)}{(t-s)^\alpha}ds \right| &\leq \frac{M_1}{\Gamma(1-\alpha)}\int_0^t s^{\sigma-1}(t-s)^{-\alpha}ds \nonumber\\
            &\leq \frac{M_1}{\Gamma(1-\alpha)}t^{\sigma-\alpha}B(\sigma,1-\alpha)\nonumber\\
            &\leq M_2t^{\sigma-\alpha},
        \end{align}
        where $M_2:=M_1 \frac{B(\sigma,1-\sigma)}{\Gamma(1-\alpha)}$. On the other hand, due to $h'(t) = O(t^{\eta})$ as $t \to \infty$, we can find a $M_3 > 0$ and $T > 1$ large enough so that $|h'(t)| \leq M_3t^{\eta}$ for all $t \geq T$. From this, for any $t > T$, we obtain 
        \begin{align}\label{4.3.40}
            \left| \frac{1}{\Gamma(1-\alpha)}\int_0^t \frac{ (s+1)^{\sigma}h'(s)}{(t-s)^\alpha}ds \right| \leq & \; \frac{1}{\Gamma(1-\alpha)}\int_0^T \frac{ (s+1)^{\sigma}|h'(s)|}{(t-s)^\alpha}ds \nonumber\\
            &+ \frac{1}{\Gamma(1-\alpha)}\int_T^t \frac{ (s+1)^{\sigma}|h'(s)|}{(t-s)^\alpha}ds\nonumber\\
            \leq &\;\frac{1}{\Gamma(1-\alpha)}\frac{(T+1)^\sigma}{(t-T)^\alpha}\int_0^T|h'(s)|ds \nonumber \\
            &+\frac{M_3(t+1)^\sigma}{\Gamma(1-\alpha)}\int_T^t s^{\eta}(t-s)^{-\alpha}ds\nonumber \\
            \leq&\; M_4(t-T)^{-\alpha} + \frac{M_3(t+1)^\sigma}{\Gamma(1-\alpha)}\int_0^t s^{\eta}(t-s)^{-\alpha}ds\nonumber \\
            =&\;M_4(t-T)^{-\alpha} + \frac{M_3(t+1)^\sigma}{\Gamma(1-\alpha)}t^{\eta-\alpha+1}B(1+ \eta,1-\alpha) \nonumber \\
            \leq &\;M_4(t-T)^{-\alpha} + M_5 t^{\eta + \sigma-\alpha+1},
        \end{align}
        where $M_4:=\frac{1}{\Gamma(1-\alpha)}(T+1)^\sigma\int_0^T|h'(s)|ds$, $M_5:=\frac{2^\sigma M_3 B(\eta+1,1-\alpha)}{\Gamma(1-\alpha)}$.
        Because $\alpha-\beta < \sigma < \alpha \leq 1$, $-1< \eta < \alpha-\sigma-1$, from \eqref{4.3.38}, \eqref{4.3.39}, and \eqref{4.3.40}, it shows that $g(t) = O(t^{\eta + \sigma-\alpha+1})$ as $t \to \infty$. Now, by \cite[Theorem 5.2]{Thinh}, we conclude that the asymptotic behavior of an arbitrary non-trivial solution of the system  \eqref{bt3}--\eqref{dkbt3} has the form $O(t^{\max\{-\beta, \eta + \sigma-\alpha+1\}})$ at infinity.
    \end{proof}
   \begin{example}
        Consider the system \eqref{bt4}--\eqref{dkbt4} with $\alpha = 2/3$, $\beta = 1/2$, $a=2$, $b=1$, and $g(t) = ^C D^{2/3} (t+1)^{1/3}\sin((t+1)^{1/4})$. The system is rewritten as
    \begin{align}\label{4.3.41}
        ^C D^{2/3}_{0^+}x(t) + 2 \, ^C D^{1/2}_{0^+}x(t) = -x(t) +  ^C D^{2/3} (t+1)^{1/3}\sin((t+1)^{1/4}), \quad t > 0.
    \end{align}
    In this case, we have $\sigma = 1/3$ and $h(t) = \sin ((t+1)^{1/4})$, $t\geq 0$. It is obvious to see that $\alpha-\beta = 1/6 <1/3=\sigma < 2/3=\alpha$, $h \in C^1[0, \infty)$, $\limsup_{t \to \infty} h(t) = 1 > 0$, and $\liminf_{t \to \infty} h(t) = -1 < 0$. Moreover, $h'(t) = \frac{1}{4}(t+1)^{-3/4} \cos((t+1)^{1/4})$, thus $|h'(t)| < \frac{1}{4}t^{-3/4}$ for all $t > 0$. From this, we have $h' = O(t^{-3/4})$ as $t \to \infty$. Notice that $-1 < -3/4 = \eta < -2/3 = \alpha - \sigma -1$, it deduces that $g$ satisfies the assumption (B)'' in Theorem  \ref{dl4.4}. Therefore, for any $x_0\in \mathbb R$, the non-trivial solution $\varphi(\cdot,x_0)$ of \eqref{4.3.41} is oscillatory and $\lim_{t\to\infty}\varphi(t,x_0)=0$.  Figure 3 shows the orbit of the solution with the initial condition $x(0)=-0.5$ on the interval $[0,150]$.
    \begin{figure}
		\begin{center}
		\includegraphics[scale=1]{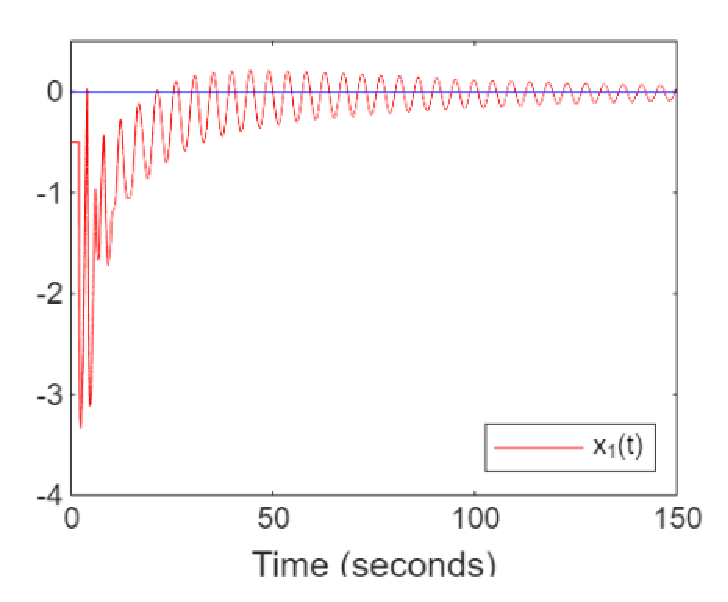}
		\end{center}
		\begin{center}
			\caption{The orbit of the solution to equation \eqref{4.3.41} with the initial condition $x(0) = -0.5$ on the interval $[0,150]$.}
		\end{center}
\end{figure}
    \end{example}
    \subsection{Asymptotic behavior of oscillatory solutions of fractional-order nonlinear equations}
    Consider the following fractional differential equation 
    \begin{align}
        ^C D^\alpha_{0^+}x(t) + a \, ^C D^\beta_{0^+}x(t)& = -bx(t) + \sum_{k=1}^m q_k(t)f_k(x(t)) + g(t),\;t>0,\label{bt4}\\
         x(0)& = x_0\in \mathbb R,\label{dkbt4}
    \end{align}
   where $0 < \beta < \alpha \leq 1$, $a,b$ are positive constants, $q_k :[0,\infty) \to \mathbb R$, $k=1,\ldots,m$, are continuous, bounded and eventually non-negative. Suppose that $g : [0,\infty) \to \mathbb R$ is continuous and satisfies the condition (B)'' in Theorem \ref{dl4.4} and $f_k : \mathbb R \to \mathbb R$, $k=1,\dots,m$, are continuous such that the assumptions below are true.
    \begin{itemize}
        \item [(F1)] $f_k(0) = 0$.
        \item [(F2)] $f_k$ is locally Lipschitz continuous at the origin and $\lim_{r \to 0} \ell_{f_k}(r) = 0$, where
        \begin{align*}
            \ell_{f_k}(r) = \sup_{|x|,|\hat x| \leq r, x \neq \hat x}\frac{|f_k(x)-f_k(\hat x)|}{|x- \hat x|}.
        \end{align*}
        \item [(F3)] $xf_k(x) \leq 0$ for all $x \in \mathbb R$.
    \end{itemize}
    \begin{theorem}\label{dl4.6} Consider the initial value problem \eqref{bt4}--\eqref{dkbt4}. Then for every $\epsilon > 0$ small enough, there exit positive parameters $\delta_1 = \delta_1(\epsilon), \delta_2 = \delta_2(\epsilon) > 0$ such that if $|g(t)| \leq \delta_1$ for all $t \geq 0$ and $|x_0| < \delta_2$, the solution $\varphi(\cdot,x_0)$ is oscillatory and $|\varphi(t,x_0)|\leq \varepsilon$ for all $t\geq 0$. Moreover, $\lim_{t\to\infty}\varphi(t,x_0)=0$.
    \end{theorem}
    \begin{proof}
         Define $F(t,x) = -bx + \sum_{k=1}^m q_k(t)f_k(x)$ for all $t\geq 0$, $x\in\mathbb R$. Since $q_k$, $k=1,\dots,m$, are eventually non-negative, it follows from (F3) that the function $F$ satisfies the condition (A) in Theorem  \ref{dl1}. Moreover, the condition (B)' in Corollary \ref{hq1} is verified by the function $g$. Hence, by Corollary \ref{hq1}, for any $x_0\in \mathbb R$, the solution $\varphi(\cdot,x_0)$ of the system \eqref{bt3}--\eqref{dkbt3} (if it exists globally) is oscillatory. 
         
         We next show that for the initial condition $x_0$ small enough and the inhomogeneous term $g$ small uniformly on $[0,\infty)$, the solution $\varphi(\cdot,x_0)$ exists globally and converges to the origin. The approach proposed in \cite[Theorem 5.7]{Thinh} will be applied to do that. We first recall some important properties of the function $G^{\kappa}_{\alpha,\beta;a,b}(t)$ given by
        \begin{align*}
            G^{\kappa}_{\alpha,\beta;a,b}(t) := \mathcal L^{-1}\left(\frac{s^{\kappa-1}}{s^\alpha + a s^\beta +b}\right)(t) \quad t\geq 0, \kappa \in \{\alpha,\beta,1\}. 
        \end{align*}
        Taking into account \cite[Proposition 4.4 (ii) and (iii)]{Thinh}, there is a positive constant $K > 0$ such that
        \begin{align}\label{4.4.44}
            |G^{1}_{\alpha,\beta;a,b}(t)| \leq \frac{K}{t^{1-\alpha}},\; t\in (0,1),\; |G^{1}_{\alpha,\beta;a,b}(t) \leq \frac{K}{t^{1 + \beta}},\; t\geq 1.
        \end{align}
       Furthermore, using the same arguments as in the proof of \cite[Proposition 4.4 (v)]{Thinh}, for any $\gamma \in (0,1)$, we can find a constant $K_\gamma > 0$ with
       \begin{align}\label{4.4.45}
           \sup_{t\geq 0}t^\gamma \int_0^t |G^{1}_{\alpha,\beta;a,b}(t-s)|s^{-\gamma}ds \leq K_\gamma.
       \end{align}
       We conclude that
       \begin{align*}
           \lim_{t \to \infty}\int_0^t G^{1}_{\alpha,\beta;a,b}(t-s)g(s)ds = 0.
       \end{align*}
       Indeed, it deduces from the proof of Theorem \ref{dl4.4} that $g = O(t^{-\gamma})$ as $t \to \infty$, here $\gamma =-(\eta + \sigma-\alpha+1) > 0$. Thus, there is a constant $M > 0$ and $t_0 > 0$ large enough satisfying
       \begin{align}\label{4.4.46}
           |g(t)| \leq \frac{M}{t^\gamma}, \; t\geq t_0.
       \end{align}
       Then, for $t > t_0 + 1$, we have  
       \begin{align}\label{I}
           \left| \int_0^t G^{1}_{\alpha,\beta;a,b}(t-s)g(s)ds \right| &\leq \int_0^{t_0} |G^{1}_{\alpha,\beta;a,b}(t-s)|\,|g(s)|ds + \int_{t_0}^{t} |G^{1}_{\alpha,\beta;a,b}(t-s)|\,|g(s)|ds \nonumber \\
           &= I_1(t) + I_2(t) .
       \end{align}
       Let $s \in [0, t_0]$. By \eqref{4.4.44}, we estimate 
       \begin{align*}
           |G^{1}_{\alpha,\beta;a,b}(t-s)| \leq K(t-s)^{-1-\beta} \leq K(t-t_0)^{-1-\beta}.
       \end{align*}
       From this, 
       \begin{align}\label{I1}
           I_1(t) \leq  K(t-t_0)^{-1-\beta}\int_0^{t_0}|g(s)|ds = K_1(t-t_0)^{-1-\beta},
       \end{align}
      where $K_1 = K \int_0^{t_0}|g(s)|ds.$ To deal with $I_2(t)$, by virtue of \eqref{4.4.45} and \eqref{4.4.46}, we see 
       \begin{align}\label{I2}
           I_2(t) &\leq M\int_{t_0}^{t} |G^{1}_{\alpha,\beta;a,b}(t-s)|\,s^{-\gamma}ds\nonumber\\
           &\leq M \int_{0}^{t} |G^{1}_{\alpha,\beta;a,b}(t-s)|\,s^{-\gamma}ds\nonumber\\
           &\leq M K_\gamma t^{-\gamma} = K_2 t^{-\gamma}.
       \end{align}
       By combining \eqref{I}, \eqref{I1} and \eqref{I2}, for any $t > t_0 + 1$, it implies that  
       \begin{align*}
           \left| \int_0^t G^{1}_{\alpha,\beta;a,b}(t-s)g(s)ds \right| \leq K_1(t-t_0)^{-1-\beta} + K_2 t^{-\gamma},
       \end{align*}
       and thus
        \begin{align}\label{4.4.50}
           \lim_{t \to \infty}\int_0^t G^{1}_{\alpha,\beta;a,b}(t-s)g(s)ds = 0.
       \end{align}
       Fix $\hat\epsilon > 0$ small enough such that $f_k, k=1,\ldots,m$ are Lipchitz continuous on $B_{\hat\epsilon}(0):=\{x\in\mathbb R:|x|\leq \hat{\varepsilon}\}$. Take $C > 0$ such that $|q_k(t)| \leq C,\; k=1,\ldots,m$ for all $t \geq 0$ and set
       \begin{align*}
           h_1(x): = - C\sum_{k=1}^m |f_k(x)|\,\,\,\text{and}\,\,\, h_2(x): =   C\sum_{k=1}^m |f_k(x)|.
       \end{align*}
       It is easy to see that $h_1$ and $h_2$ satisfy the assumptions (F1) and (F2).
       
 Consider the initial value problems 
       \begin{align}
        ^C D^\alpha_{0^+}x(t) + a \, ^C D^\beta_{0^+}x(t) &= -bx(t) + h_1(x) + g(t), \quad t> 0, \label{ivp_t11}\\
        x(0) &= x_0,\label{ivp_t12}
    \end{align}
        and
       \begin{align}
        ^C D^\alpha_{0^+}x(t) + a \, ^C D^\beta_{0^+}x(t) &= -bx(t) + h_2(x) + g(t), \quad t> 0, \label{ivp_t21}\\
        x(0) &= x_0.\label{ivp_t22}
    \end{align}
    Choosing $\varepsilon>0$ be small enough, for example, $\varepsilon\leq \hat{\varepsilon}$ and
    \[
    \max_{1\leq k\leq m}\ell_{f_k}(\varepsilon)\int_0^\infty |G^1_{\alpha,\beta;a,b}(t)|dt<\frac{1}{2}.
    \]
    In light of \cite[Theorem 5.7]{Thinh}, we can find positive parameters $ \delta_1, \delta_2>0$ such that if $|g(t)| \leq \delta_1$ for all $t \geq 0$ and $|x_0| \leq \delta_2$ then $|\varphi_1(t,x_0)|,  |\varphi_2(t,x_0)|\leq \epsilon$ for all $t \geq 0$ and $\lim_{t \to \infty} \varphi_1(t,x_0) =\lim_{t \to \infty} \varphi_2(t,x_0)= 0$, here $\varphi_1(\cdot,x_0), \varphi_2(\cdot,x_0)$ are the solutions of the systems \eqref{ivp_t11}--\eqref{ivp_t12} and \eqref{ivp_t21}--\eqref{ivp_t22}, respectively. On the other hand, according to Theorem \ref{ssnh} and Corollary \ref{sslh}, we claim that
       \begin{align}\label{4.4.51}
           \varphi_1(t,x_0) \leq \varphi(t,x_0) \leq \varphi_2(t,x_0), \quad t \geq 0.
       \end{align}
       This implies that the solution $\varphi(\cdot,x_0)$ exists globally and $\lim_{t \to \infty} \varphi(t,x_0)= 0$. The proof is complete.
    \end{proof}
    \begin{example}
        Consider the system \eqref{bt4}--\eqref{dkbt4} with $\alpha = 2/3$, $\beta = 1/2$, $a=2$, $b=1$, $m=1$, $q_1(t) = 1,\; t \geq 0$, $f_1(x) =-x^3,\; x \in \mathbb R$  and $g(t) = \frac{1}{100}{}^C D^{2/3} (t+1)^{1/3}\sin((t+1)^{1/4})$. The system is rewritten as
    \begin{align}\label{4.3.59}
        ^C D^{2/3}_{0^+}x(t) + 2 \, ^C D^{1/2}_{0^+}x(t) = -x(t) -2x^2(t)+ \frac{1}{100}{} ^C D^{2/3} (t+1)^{1/3}\sin((t+1)^{1/4}), \quad t > 0. 
    \end{align}
    It is obvious to see that the assumptions in Theorem \ref{dl4.6} are satisfied. Thus, for the initial condition $x_0$ is small enough, the solution $\varphi(\cdot,x_0)$ is oscillatory and $\lim_{t\to\infty}\varphi(t,x_0)=0$. Figure 4 depicts the orbit of the solution to equation \eqref{4.3.59} with the initial condition $x(0) = 0.6$ on the interval $[0,150]$.
     \begin{figure}
		\begin{center}
		\includegraphics[scale=0.7]{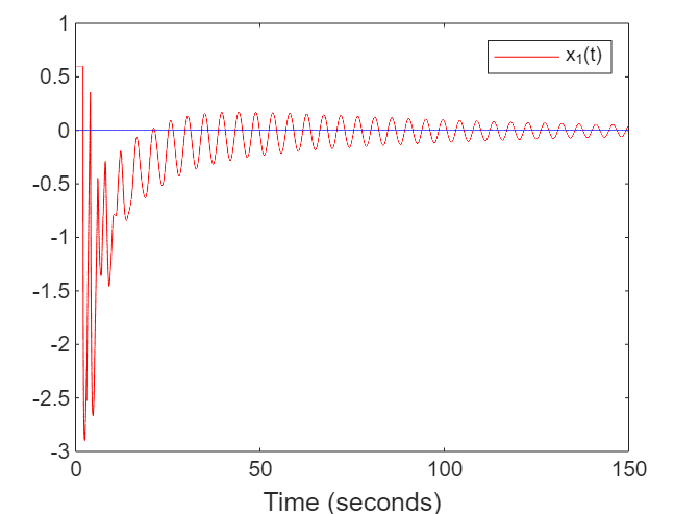}
		\end{center}
		\begin{center}
			\caption{The orbit of the solution to equation \eqref{4.3.59} with the initial condition $x(0) = 0.6$ on the interval $[0,150]$.}
		\end{center}
\end{figure}
    \end{example}
\section{On the asymptotic behavior of non-oscillatory solutions of multi-order fractional differential equations}
Consider the equation
\begin{align}\label{btdk}
    &{}^C D^{\alpha}_{0^+}x(t) + a\,{}^C D^{\beta}_{0^+}x(t) + q(t)f(x) + \mathrm{sgn}(x(t))g(t) = 0, \quad t > 0, \\
    & \hspace{1 cm} x(0) =x_0, x^{(i)}(0) = x_i, i =1,2,\ldots, n-1, \label{gtd}
\end{align}
where $n-1 < \beta < \alpha \leq  \lceil  \alpha \rceil = n$, $n \geq 3$, 
$a, x_i \in \mathbb R, i=0,1,2, \ldots,n-1 $, $a > 0$, $f : \mathbb R \to \mathbb R$, $q : [0, \infty) \to \mathbb R$, and $g : [0, \infty) \to \mathbb R$ are continuous.

In this section, the following restrictions are imposed on the functions $f(\cdot)$, $q(\cdot)$, and $g(\cdot)$. 
\begin{itemize}
    \item [(G1)] $xf(x) > 0$ for all $x \neq 0$.
    \item [(G2)] There are positive constants $c_1, \eta > 0$ such that $|f(x)| \geq c_1 |x|^\eta$ for all $|x| \geq 1$.
    \item [(G3)] There exist positive constants $c_2 > 0$, $\mu \in (0, \frac{n-2}{n-\beta})$  with $|f(x)| \geq c_2 |x|^\mu$ for all $|x| < 1$.
    \item [(H)] There are $T_1,c_3>0$ so that $q(t) \geq c_3 > 0$ for all $t \geq T_1.$
    \item [(K)] There are $\sigma \in (n-\beta,1)$, $c_4\in (0,\infty)$  and $T_2>0$ such that $$g(t) \geq c_4 t^{n-\beta-\sigma}, \quad t \geq T_2.$$
\end{itemize}
Under the assumptions mentioned above, we aim to study the asymptotic behavior of the non-oscillatory solutions of the higher-order fractional differential equation \eqref{btdk}. Our main contributions are as follows.
\begin{theorem}\label{dl3}
     Assume that the assumptions $\textup{(G1), (G2), (H), (K)}$ are true and $n$ is odd. Let $x(\cdot)$ be a non-oscillatory solution of the initial value problem \eqref{btdk}--\eqref{gtd}. Then, $$x(t)=O(t^{\alpha-n})\;\;\text{as}\;t\to\infty.$$
\end{theorem}
\begin{theorem}\label{dl4}
    Assume that the assumptions $\textup{(G1), (G2), (G3), (H), (K)}$ are true and $n$ is even. Let $x(\cdot)$ be a non-oscillatory solution of the initial value problem \eqref{btdk}--\eqref{gtd}. Then, $$x(t)=O(t^{\alpha-n})\;\;\text{as}\;t\to\infty.$$
\end{theorem}
The proof of the above results is long and requires much preparation. To make the presentation clear and easy to follow, we have divided it into a sequence of lemmas.

Throughout this part, the following notations will be used:
\begin{align}\label{x}
    \begin{split}
        y_1(t) = &I^{n-\alpha}_{0^+}x(t) + a\,I^{n-\beta}_{0^+}x(t)\\
         = &I^{n-\alpha}_{0^+}\left(x(t) + a\,I^{\alpha-\beta}_{0^+}x(t)\right), \\
        y_2(t)= &^C D^{\alpha-n+1}_{0^+}x(t) + a \,^C D^{\beta-n+1}_{0^+}x(t)+\frac{x_0}{\Gamma(n-\alpha)}t^{n-\alpha-1}+\frac{a\,x_0}{\Gamma(n-\beta)}t^{n-\beta-1},\\
        y_i(t) =& ^C D^{\alpha-n -1 +i}_{0^+}x(t) + a\,^C D^{\beta-n-1 +i}_{0^+}x(t) \\
        &+ \sum_{k=0}^{i-2}\frac{x_k\prod_{j=0}^{i-2-k}(n-\alpha-i+2+k+j)}{\Gamma(n-\alpha+1)}t^{n-\alpha-i+1+k} \\
        &+\sum_{k=0}^{i-2}\frac{a\,x_k\prod_{j=0}^{i-2-k}(n-\beta-i+2+k+j)}{\Gamma(n-\beta+1)}t^{n-\beta-i+1+k},\; i = 3, \ldots,n,\;t>0.
    \end{split}    
\end{align}
\begin{lemma}\label{bd2}
     Let $x(\cdot)$ be a solution of the system  \eqref{btdk}--\eqref{gtd} and $y_i, i = 1, 2,3, \ldots, n$ are defined as in \eqref{x}. Then, $y_1 \in AC[0, \infty)$, $y_i \in C(0,\infty),\; i=2,3,\ldots,n$, and
    \begin{align}\label{hptvp}
        \begin{split}
            y_i'(t) =& y_{i+1}(t),\; t>0,\; i=1, 2, \ldots, n-1, \\
            y_n'(t) = &-q(t)f(x(t))-\mathrm{sgn}(x(t))g(t)+ \sum_{k=0}^{n-1}\frac{a_k\prod_{j=0}^{n-1-k}(1-\alpha+k+j)}{\Gamma(n-\alpha+1)}t^{-\alpha+k}\\
            &+\sum_{k=0}^{n-1}\frac{aa_k\prod_{j=0}^{n-1-k}(1-\beta+k+j)}{\Gamma(n-\beta+1)}t^{-\beta+k},\; t>0.
        \end{split}    
    \end{align}
    In addition, if (G1), (H), (K) are satisfied and $x(\cdot)$ is eventually positive, then $y_n(\cdot)$ is monotonically decreasing and $y_i$,  $i=1,2,\ldots,n-1,$ are monotonic on the interval $[T,\infty)$ with some $T>0$ large enough.
\end{lemma}
\begin{proof}
    By a simple computation, for $k \in \mathbb N \cup \{0\}$,  we have
    \begin{align*}
        I^{n-\alpha}_{0^+}t^k = &\frac{1}{\Gamma(n-\alpha)}\int_0^t (t-s)^{n-\alpha-1}s^kds \\
        &= \frac{1}{\Gamma(n-\alpha)}\int_0^1t^{n-\alpha-1}(1-u)^{n-\alpha-1}t^ku^ktdu\\
        &=\frac{1}{\Gamma(n-\alpha)} t^{n-\alpha+k}B(k+1,n-\alpha) \\
        &= \frac{k!}{\Gamma(n-\alpha+1+k)}t^{n-\alpha+k},\;t>0,
    \end{align*}
  where $B(\cdot,\cdot)$ is the Beta function, and thus
    \begin{align*}
        I^{n-\beta}_{0^+}t^k = \frac{k!}{\Gamma(n-\beta+1+k)}t^{n-\beta+k},\; \forall k \in \mathbb N \cup \{0\},\; t>0.
    \end{align*}   
    To simplify notation, we use the symbol $D$  instead of $\frac{d}{dt}$. First, for any $t > 0$, then
    \begin{align*}
        y_1'(t) &= D\,I^{n-\alpha}_{0^+}x(t) + a\, D\, I^{n-\beta}_{0^+}x(t) \\
        &= D\, I^{1-(\alpha-n+1)}_{0^+}(x(t)-x_0) + a\,  D\,I^{1-(\beta-n+1)}_{0^+}(x(t)-x_0) \\
        &\hspace{7cm}+ D\left( I^{n-\alpha}_{0^+}x_0 + a\,I^{n-\beta}_{0^+}x_0 \right)\\
        &={^CD^{\alpha-n+1}}_{0^+}x(t) + a \,^C D^{\beta-n+1}_{0^+}x(t)+\frac{a_0}{\Gamma(n-\alpha)}t^{n-\alpha-1}+\frac{a\,a_0}{\Gamma(n-\beta)}t^{n-\beta-1}\\
        &= y_2(t).
    \end{align*}
    Next, by direct computation, it is not difficult to check that
    \begin{align*}
        y_2'(t) =&D\left(^C D^{\alpha-n+1}_{0^+}x(t)\right) + a\, D\left(^C D^{\beta-n+1}_{0^+}x(t)\right) \\
        &\hspace{5cm} +\frac{x_0(n-\alpha-1)}{\Gamma(n-\alpha)}t^{n-\alpha-2} +\frac{ax_0(n-\beta-1)}{\Gamma(n-\beta)}t^{n-\beta-2}\\
        =& D\left(D I^{1-(\alpha-n+1)}_{0^+}(x(t)-x_0)\right) + a\, D\left(D I^{1-(\beta-n+1)}_{0^+}(x(t)-x_0)\right)\\
        &\hspace{5cm} +\frac{x_0(n-\alpha-1)}{\Gamma(n-\alpha)}t^{n-\alpha-2} +\frac{ax_0(n-\beta-1)}{\Gamma(n-\beta)}t^{n-\beta-2}\\
        =&D^{(2)}\left(I^{2-(\alpha-n+2)}_{0^+}(x(t)-x_0-x_1t)\right) + a\,D^{(2)}\left(I^{2-(\beta-n+2)}_{0^+}(x(t)-x_0-x_1t) \right)\\
        &+D^{(2)}\left(I^{n-\alpha}_{0^+}a_1t\right)+ a\,D^{(2)}\left(I^{n-\beta}_{0^+}x_1t\right)+\frac{x_0(n-\alpha-1)}{\Gamma(n-\alpha)}t^{n-\alpha-2} +\frac{aa_0(n-\beta-1)}{\Gamma(n-\beta)}t^{n-\beta-2}\\
        =&^C D^{\alpha-n+2}_{0^+}x(t) + a\,^C D^{\beta-n+2}_{0^+}x(t) + \frac{x_1}{\Gamma(n-\alpha)}t^{n-\alpha-1} +  \frac{ax_1}{\Gamma(n-\beta)}t^{n-\beta-1}\\
        &\hspace{4cm} +\frac{x_0(n-\alpha-1)(n-\alpha)}{\Gamma(n-\alpha+1)}t^{n-\alpha-2} +\frac{ax_0(n-\beta-1)(n-\beta)}{\Gamma(n-\beta+1)}t^{n-\beta-2}\\
        =& {^CD^{\alpha-n-1+3}}_{0^+}x(t) + a\, ^C D^{\beta-n-1+3}_{0^+}x(t)\\
        &+ \sum_{k=0}^{(3-2)}\frac{x_k\prod_{j=0}^{3-2-k}(n-\alpha-3+2+k+j)}{\Gamma(n-\alpha+1)}t^{n-\alpha-3+1+k} \\ &+\sum_{k=0}^{(3-2)}\frac{ax_k\prod_{j=0}^{3-2-k}(n-\beta-3+2+k+j)}{\Gamma(n-\beta+1)}t^{n-\beta-3+1+k} \\
        =&x_3(t)\;\;\text{for all}\;t>0.
    \end{align*}
    Using the same arguments, for $i= 3, \ldots, n-1$ and $t>0$, we also obtain
    \begin{align*}
        y_i'(t) =&D\, \left(^C D^{\alpha-n-1+i}_{0^+}x(t)\right) + a D\,\left(^C D^{\beta-n-1+i}_{0^+}x(t)\right) \\
        &+ \sum_{k=0}^{i-2}\frac{x_k(n-\alpha-i+1+k)\prod_{j=0}^{i-2-k}(n-\alpha-i+2+k+j)}{\Gamma(n-\alpha+1)}t^{n-\alpha-i+k} \\
        &+\sum_{k=0}^{i-2}\frac{ax_k(n-\beta-i+1+k)\prod_{j=0}^{i-2-k}(n-\beta-i+2+k+j)}{\Gamma(n-\beta+1)}t^{n-\beta-i+k}\\
        =&D\left( D^{(i-1)}I^{(i-1-(\alpha-n-1+i))}\left(x(t) - \sum_k^{i-2}\frac{x_k}{k!}t^k\right) \right) \\
        &\hspace{3cm}+ a\,D\left( D^{(i-1)}I^{(i-1-(\beta-n-1+i))}\left(x(t) - \sum_k^{i-2}\frac{x_k}{k!}t^k\right) \right)\\
        &+\sum_{k=0}^{i-2}\frac{x_k(n-\alpha+k-i+1)\prod_{j=1}^{i-1-k}(n-\alpha-i+1+k+j)}{\Gamma(n-\alpha+1)}t^{n-\alpha-i+k}\\
        &+\sum_{k=0}^{i-2}\frac{ax_k(n-\beta+k-i+1)\prod_{j=1}^{i-1-k}(n-\beta-i+1+k+j)}{\Gamma(n-\beta+1)}t^{n-\beta-i+k}\\
        =& D^{(i)} I^{i - (\alpha-n + i)}_{0^+}\left(x(t)-\sum_{k=0}^{i-1}\frac{x_k}{k!}t^k\right) + a D^{(i)} I^{i - (\beta-n + i)}_{0^+}\left(x(t)-\sum_{k=0}^{i-1}\frac{x_k}{k!}t^k\right)\\
        &\hspace{3cm}+ D^{(i)}\left(I^{n-\alpha}_{0^+}\left(\frac{x_{i-1}}{(i-1)!}t^{i-1} \right) +I^{n-\beta}_{0^+}\left(\frac{ax_{i-1}}{(i-1)!}t^{i-1} \right)\right) \\
        &+\sum_{k=0}^{i-2}\frac{x_k\prod_{j=0}^{i-1-k}(n-\alpha-i+1+k+j)}{\Gamma(n-\alpha+1)}t^{n-\alpha-i+k}\\
        &\hspace{3cm}+ \sum_{k=0}^{i-2}\frac{ax_k\prod_{j=0}^{i-1-k}(n-\beta-i+1+k+j)}{\Gamma(n-\beta+1)}t^{n-\beta-i+k}\\
        =&^C D^{\alpha-n + i}_{0^+}x(t) + a ^C D^{\beta-n + i}_{0^+}x(t) + \frac{x_{i-1}}{\Gamma(n-\alpha)}t^{n-\alpha-1} + \frac{ax_{i-1}}{\Gamma(n-\beta)}t^{n-\beta-1}\\
        &+\sum_{k=0}^{i-2}\frac{x_k\prod_{j=0}^{i-1-k}(n-\alpha-i+1+k+j)}{\Gamma(n-\alpha+1)}t^{n-\alpha-i+k}\\
        &+ \sum_{k=0}^{i-2}\frac{ax_k\prod_{j=0}^{i-1-k}(n-\beta-i+1+k+j)}{\Gamma(n-\beta+1)}t^{n-\beta-i+k} \\
        =&^C D^{\alpha-n + i}_{0^+}x(t) + a ^C D^{\beta-n + i}_{0^+}x(t) +  \sum_{k=0}^{i-1}\frac{x_k\prod_{j=0}^{i-1-k}(n-\alpha-i+1+k+j)}{\Gamma(n-\alpha+1)}t^{n-\alpha-i+k}\\
        &\hspace{3cm}+\sum_{k=0}^{i-1}\frac{ax_k\prod_{j=0}^{i-1-k}(n-\beta-i+1+k+j)}{\Gamma(n-\beta+1)}t^{n-\beta-i+k}\\
        =&^C D^{\alpha-n -1 + (i+1)}_{0^+}x(t)+ a ^C D^{\beta-n -1 + (i+1)}_{0^+}x(t) \\ &+\sum_{k=0}^{(i+1)-2}\frac{x_k\prod_{j=0}^{(i+1)-2-k}(n-\alpha-(i+1)+2+k+j)}{\Gamma(n-\alpha+1)}t^{n-\alpha-(i+1)+1+k}\\
        &+\sum_{k=0}^{(i+1)-2}\frac{ax_k\prod_{j=0}^{(i+1)-2-k}(n-\beta-(i+1)+2+k+j)}{\Gamma(n-\beta+1)}t^{n-\beta-(i+1)+1+k}\\
        =&x_{i+1}(t).
    \end{align*}
    Finally, for $t > 0$, from the definition of $y_n(\cdot)$ and the observations shown
    above, then 
    \begin{align}\label{23}
        y_n'(t)=&^C D^{\alpha}_{0^+}x(t) + a ^C D^\beta_{0^+}x(t) + \sum_{k=0}^{n-1}\frac{x_k\prod_{j=0}^{n-1-k}(1-\alpha+k+j)}{\Gamma(n-\alpha+1)}t^{-\alpha+k}\nonumber\\
        &+\sum_{k=0}^{n-1}\frac{ax_k\prod_{j=0}^{n-1-k}(1-\beta+k+j)}{\Gamma(n-\beta+1)}t^{-\beta+k}\nonumber\\
        =&-q(t)f(x(t))-\mathrm{sgn}(x(t))g(t)+ \sum_{k=0}^{n-1}\frac{x_k\prod_{j=0}^{n-1-k}(1-\alpha+k+j)}{\Gamma(n-\alpha+1)}t^{-\alpha+k}\nonumber\\
        &+\sum_{k=0}^{n-1}\frac{ax_k\prod_{j=0}^{n-1-k}(1-\beta+k+j)}{\Gamma(n-\beta+1)}t^{-\beta+k},
    \end{align}
    which implies that \eqref{hptvp} is verified for all 
    $t > 0$. On the other hand, due to $\beta-n-1+i< \alpha-n-1+i \leq \alpha-1 \leq n-1$ for $i=2,3, \ldots,n$ and $x \in C^{n-1}[0,\infty)$, it follows from \cite[Theorem 8]{Cong} that $^C D^{\alpha-n-1+i}_{0^+}x$, $^C D^{\beta-n-1+i}_{0^+}x \in C[0,\infty)$, $i = 2, 3, \ldots, n$. Therefore, $y_i \in C(0,\infty)$, $i = 2,3,\ldots,n$. Notice that by $y_1'(t) = y_2(t)$ for all $t > 0$, $y_2(\cdot) \in C(0,\infty)$ and $y_2(\cdot)$ is integrable on finite subintervals of $[0,\infty)$, we conclude that $y_1 \in AC[0,\infty)$.

    Now, under the added assumptions $\textup{(G1), (H), (K)}$ and suppose that $x(\cdot)$ is eventually positive, we can find a parameter $t_0 > 0$ such that $x(t) >0$, $f(x(t)) > 0$, $q(t) >0$ and $g(t) \geq c_4t^{n-\beta-\sigma}$ for all $t \geq t_0$. From this together with the fact $-\alpha + k < -\beta + k \leq n-\beta-1 < 0$ for all $k=0,1,2,\ldots,n-1$, there are constants $M>0$ and $t_1 > t_0$ so that
    \begin{align}\label{24}
        &\sum_{k=0}^{n-1}\frac{x_k\prod_{j=0}^{n-1-k}(1-\alpha+k+j)}{\Gamma(n-\alpha+1)}t^{-\alpha+k}\notag\\
        &\hspace{2cm}+\sum_{k=0}^{n-1}\frac{ax_k\prod_{j=0}^{n-1-k}(1-\beta+k+j)}{\Gamma(n-\beta+1)}t^{-\beta+k}\notag\\
        &\hspace{1cm}\leq Mt^{n-\beta-1},\;\forall t \geq t_1.
    \end{align}
    Let $T = \max\left\{\left(\frac{M}{c_4}\right)^{\frac{1}{1-\sigma}}, t_1\right\}$, then 
    \begin{align}\label{25}
        Mt^{n-\beta-1} < c_4t^{n-\beta-\sigma} < g(t),\;\forall t>T.
    \end{align}
    By combining \eqref{23}, \eqref{24}, \eqref{25} and the fact $\mathrm{sgn}(x(t)) = 1$, $q(t)f(x(t)) > 0$ for all $t > T$ yields $y'_n(t) < 0$ on $[T,\infty)$. This together with the relation $y'_i(t) = y_{i+1}(t)$, $t>0$, $i=1,2,\ldots,n-1$ leads to that $y_n(\cdot)$ is monotonically decreasing and $y_i$, $i=1,2,\ldots,n-1$ are monotonic on $[T,\infty)$. The proof finishes.
\end{proof}
\begin{lemma}\label{bd3}
    Consider the equation \eqref{btdk}. Suppose that $\textup{(G1), (H), (K)}$ are true. If $x(\cdot)$ is an eventually positive solution and $y_1(\cdot)$ is eventually negative, then $\lim_{t \to \infty} y_1(t) = 0$ and $y_2(\cdot)$ is eventually positive.
\end{lemma}
\begin{proof}
    Let $x(\cdot)$ be a solution of the equation  \eqref{btdk} on $[0,\infty)$. Assume that it is eventually positive and $y_1(\cdot)$ is eventually negative. From this, there is a $t_0 > 0$ large enough such  $x(t) > 0$, $y_1(t) < 0$ for all $t \geq t_0$. Taking $M: = \int_0^{t_0}|x(s)|ds$. For any $t > t_0$, it is easy to see  
    \begin{align}\label{26}
        I^{n-\alpha}_{0^+}x(t) &= \frac{1}{\Gamma(n-\alpha)}\int_{0}^t (t-s)^{n-\alpha-1}x(s)ds\nonumber\\
        & = \frac{1}{\Gamma(n-\alpha)}\int_{0}^{t_0} (t-s)^{n-\alpha-1}x(s)ds + \frac{1}{\Gamma(n-\alpha)}\int_{t_0}^t (t-s)^{n-\alpha-1}x(s)ds \nonumber\\
        &\geq \frac{1}{\Gamma(n-\alpha)}\int_{0}^{t_0} (t-s)^{n-\alpha-1}x(s)ds.
    \end{align}
    Notice that $(t-s)^{n-\alpha-1}x(s) \geq -(t-t_0)^{n-\alpha-1}|x(s)| $ for all $s \in [0,t_0]$ and $t>t_0$. This together with \eqref{26} lead to 
    \begin{align}\label{27}
        I^{n-\alpha}_{0^+}x(t) \geq \frac{-M}{\Gamma(n-\alpha)}(t-t_0)^{n-\alpha-1}, \; \forall t > t_0.
    \end{align}
    By the same arguments, it is also true that 
    \begin{align}\label{28}
        I^{n-\beta}_{0^+}x(t) \geq \frac{-M}{\Gamma(n-\beta)}(t-t_0)^{n-\beta-1}, \; \forall t > t_0.
    \end{align}
    From \eqref{x}, \eqref{27} and \eqref{28}, we get
    \begin{align*}
       \frac{-M}{\Gamma(n-\alpha)}(t-t_0)^{n-\alpha-1} + \frac{-aM}{\Gamma(n-\beta)}(t-t_0)^{n-\beta-1} \leq x_1(t) < 0, \; \forall t > t_0,
    \end{align*}
    and thus $\lim_{t\to\infty} x_1(t) = 0$. On the other hand, it follows from Lemma \ref{bd2} that $y_1(\cdot)$ is strictly monotonic on $[T,\infty)$ for some $T>0$ large enough. Thus $y_1(\cdot)$ is strictly increasing on that interval. Due to $y'_1(t) =y_2(t)$ for all $t > 0$, we conclude that $y_2(\cdot)$ is eventually positive.
\end{proof}
\begin{lemma}\label{bd4}
Suppose that $\textup{(G1), (G2), (H), (K)}$ hold. Let $x(\cdot)$ be an eventually positive solution of \eqref{btdk} such that $y_n(\cdot)$ is also eventually positive, then 
\begin{align*}
    \limsup_{t\to \infty}x(t) < \infty.
\end{align*}
\end{lemma}
\begin{proof}
    Because $\textup{(F1), (F2), (Q), (G)}$ are true, $x(\cdot), y_n(\cdot)$ are eventually positive, there is a $t_0 > 0$ so that $x(t) >0$, $y_n(t) > 0$, $f(x(t)) > 0$, $q(t) \geq c_3 >0$ and $ g(t) \geq c_4t^{n-\beta-\sigma}$ for all $t \geq t_0$. From this, by \eqref{23}, \eqref{24} and \eqref{25}, we can find $t_1 > t_0$ so that 
    \begin{align*}
        y_n'(t) < -q(t)f(x(t)), \; \forall t > t_1.
    \end{align*}
  Moreover, 
    \begin{align*}
        y_n(t_1) &> y_n(t_1) -y_n(t) = -\int_{t_1}^t y'_n(s)ds \nonumber\\
        &>\int_{t_1}^t q(s)f(x(s))ds \nonumber\\
        &\geq c_3 \int_{t_1}^t f(x(s))ds,\; \forall t > t_1.
    \end{align*}
    Letting $t \to \infty$, then 
    \begin{equation}\label{Tuan_add1}
        \int_{t_1}^\infty f(x(s))ds < \frac{y_n(t_1)}{c_3}<\infty,
    \end{equation}
    which implies that $f(x(\cdot))$ is bounded on $[t_1, \infty)$. This together with $\textup{(G2)}$ shows
    that $x(\cdot)$ is also bounded on $[t_1, \infty)$ and thus
    \begin{align*}
        \limsup_{t\to \infty} x(t) < \infty.
    \end{align*}
\end{proof}
\begin{lemma}\label{bd5}
    Let $x(\cdot)$ be an eventually positive solution of \eqref{btdk} satisfying
    \begin{align*}
        y_2(t) \geq c_5 > 0,\; \forall t \geq T_3 >0.
    \end{align*}
    Then, 
    \begin{align*}
        \limsup_{t \to \infty}x(t) = \infty.
    \end{align*}
\end{lemma}
\begin{proof}
    From \eqref{x} and Lemma \ref{tddh}, we have 
    \begin{align}\label{37}
         ^C D^{n-\alpha}_{0^+}y_1(t) = x(t) + a I^{\alpha-\beta}_{0^+}x(t),\;\forall t>0.
    \end{align}
    Define $x^*(t) = sup_{s \in [0,t]}x(s)$ for each $t>0$, then 
    \begin{align}\label{38}
        x(t) + a I^{\alpha-\beta}_{0^+}x(t) &= x(t) + \frac{a}{\Gamma(\alpha-\beta)}\int_0^t(t-s)^{\alpha-\beta-1}x(s)ds \nonumber\\
        & \leq x^*(t) + \frac{ax^*(t)}{\Gamma(\alpha-\beta)}t^{\alpha-\beta} \nonumber\\
        &= x^*(t) \left( 1 + \frac{a}{\Gamma(\alpha-\beta)}t^{\alpha-\beta} \right),\; \forall t>0.
    \end{align}
   On the other hand, by the assumption of the lemma, there is a $t_0 > 0$ such that $x(t) > 0$, $y_2(t) \geq c_5 > 0$ for all $t \geq t_0$. 
    Define $M_1: = \int_0^{t_0} |y_2(s)|ds$. It follows from Lemma \ref{dhc} that 
    \begin{align}\label{39}
        ^C D^{n-\alpha}_{0^+}y_1(t) &= \frac{1}{\Gamma(1 + \alpha -n)}\int_0^t\frac{y_2(s)}{(t-s)^{n-\alpha}}ds \nonumber \\
        &=\frac{1}{\Gamma(1 + \alpha -n)}\left(\int_0^{t_0}\frac{y_2(s)}{(t-s)^{n-\alpha}}ds + \int_{t_0}^{t}\frac{y_2(s)}{(t-s)^{n-\alpha}}ds \right) \nonumber\\
        &\geq \frac{1}{\Gamma(1 + \alpha -n)}\left(\frac{-1}{(t-t_0)^{n-\alpha}}\int_0^{t_0}|x(s)|ds +\frac{c_5}{\alpha-n+1}(t-t_0)^{\alpha-n+1}\right)\nonumber \\
        & = \frac{1}{\Gamma(1 + \alpha -n)}\left(\frac{-M_1}{(t-t_0)^{n-\alpha}} + \frac{c_5}{\alpha-n+1}(t-t_0)^{\alpha-n+1} \right),\;\forall t > t_0.
    \end{align}
    Using \eqref{37}--\eqref{39}, for $t > t_0$, then 
    \begin{align*}
        x^*(t) \left( 1 + \frac{a}{\Gamma(\alpha-\beta)}t^{\alpha-\beta} \right) \geq \frac{1}{\Gamma(1 + \alpha -n)}\left(\frac{-M_1}{(t-t_0)^{n-\alpha}} + \frac{c_5}{\alpha-n+1}(t-t_0)^{\alpha-n+1} \right),
    \end{align*}
    which gives
    \begin{align*}
        x^*(t) \geq \frac{\frac{1}{\Gamma(1 + \alpha -n)}\left(\frac{-M_1}{(t-t_0)^{n-\alpha}} + \frac{c_5}{\alpha-n+1}(t-t_0)^{\alpha-n+1} \right)}{1 + \frac{a}{\Gamma(\alpha-\beta)}t^{\alpha-\beta}}, \; \forall t > t_0,
    \end{align*}
    and thus  
    $x^*(t) \to \infty$ as $t \to \infty$. In particular, it follows that
    \begin{align*}
        \limsup_{t \to \infty}x(t) = \infty.
    \end{align*}
\end{proof}
\begin{lemma}\label{bd6}
    Suppose that $\textup{(G1), (G2), (H), (K)}$ are true. The following statements hold.
    \begin{itemize}
        \item [(i)] If $n$ is odd, all eventually positive solutions of \eqref{btdk} satisfy
        \begin{align}\label{a}
            (-1)^{i+1}y_i(t) >0,\; i=1,2,3,\ldots,n,\,\,\,\text{for $t$ large enough}. 
        \end{align}
        \item [(ii)] If $n$ is even, all eventually positive solutions of \eqref{btdk} satisfy 
        \begin{align}\label{b}
            (-1)^{i}y_i(t) >0,\; i=1,2,3,\ldots,n,\,\,\,\text{for $t$ large enough},
        \end{align}
        or
        \begin{align}\label{c}
            y_1(t) > 0,\; y_2(t) > 0\,\,\,\text{and}\,\,\, (-1)^{i}y_i(t) >0,\; i=3,\ldots,n,\,\,\,\text{for $t$ large enough.}
        \end{align}
    \end{itemize}
\end{lemma}
\begin{proof}
    Let $x(\cdot)$ be an eventually positive solution of \eqref{btdk}. By Lemma \ref{bd2}, there is a $T>0$ such that $y'_n(t) < 0$ and $y_i$, $i=1,2,3,\ldots,n$, are strictly monotonic on $[T,\infty)$. From this, we can find $t_0 > T$ with $y'_n(t) < 0$ and $y_i(\cdot)$, $i=1,2,3,\ldots,n,$ do not change sign on $[t_0,\infty)$. We first point out that if there exits $k \in \{2,3, \ldots, n\}$ satisfying $y_{k-1}(t) y_k(t) > 0$ for all $t \geq t_0$, then  
    \begin{align*}
        y_i(t)y_{i+1}(t) > 0,\; i=1,2, \ldots, k-1,\;t \geq t_0.
    \end{align*}
    Indeed, without loss of generality, we assume $y_{k-1}(t) > 0$, $y_k(t) > 0$ for all $t \geq t_0$. By induction, we only need to prove $y_{k-2}(t) > 0$ for all $t \geq t_0$. Due to $y'_{k-1}(t) = y_k(t) > 0$ for all $t \geq t_0$, the function $y_{k-1}(\cdot)$ is positive and strictly increasing on $[t_0, \infty)$. This implies that 
    \begin{align*}
        y_{k-1}(t) \geq M_3, \quad t \geq t_1,
    \end{align*}
    for some $M_3>0$ and $t_1>t_0$ and thus
    \begin{align}\label{bs2}
        y_{k-2}(t) = \int_{t_1}^t y_{k-1}(s)ds + y_{k-2}(t_1) \geq M_3(t-t_1) +y_{k-2}(t_1),\;\forall t>t_1.
    \end{align}
    From \eqref{bs2},  $ y_{k-2}(t) > 0$ for all $t > \frac{M_3 t_1 -y_{k-2}(t_1)}{M_3}$ which together the fact $x_{k-2}(\cdot)$ does not change sign on $[t_0,\infty)$ gives $y_{k-2}(t) > 0$ for all $t \geq t_0.$

\noindent {\bf Case I:} there exits $k \in \{2,3, \ldots, n\}$ satisfying $y_{k-1}(t) y_k(t) > 0$ for all $t \geq t_0$. In this case, only the following possibilities occur.
    \begin{itemize}
        \item [(i)] If $k=n$, then 
            \begin{align}\label{(1)}
                y_i(t) > 0, \; t \geq t_0,\; i=1,2,\ldots,n,
            \end{align}
        or
            \begin{align}\label{(2)}
                y_i(t) < 0, \; t \geq t_0,\; i=1,2,\ldots,n.
            \end{align}
        \item [(ii)] If $2 \leq k \leq n-1$, then for $t \geq t_0$,
            \begin{align}\label{(3)}
                y_i(t) > 0,\;i=1,2,\ldots,k,\; y_{i}(t)y_{i+1}(t) < 0,\; i = k,\ldots,n-1,
            \end{align}
        or
            \begin{align}\label{(4)}
                y_i(t) < 0,\; i=1,2,\ldots,k,\; y_{i}(t)y_{i+1}(t) < 0,\;i = k,\ldots,n-1.
            \end{align}
    \end{itemize}

\noindent {\bf Case II:} there is not $k \in \{2,3, \ldots, n\}$ satisfying $y_{k-1}(t) y_k(t) > 0$ for all $t \geq t_0$. It is easy to check that then the following statement is true.
    \begin{itemize}
         \item [(iii)] 
            \begin{align}\label{(5)}
                (-1)^i y_i(t) > 0, \; t \geq t_0,\; i=1,2,\ldots,n,
            \end{align}
        or
            \begin{align}\label{(6)}
                (-1)^{i+1} y_i(t) > 0, \; t \geq t_0,\; i = 1, 2, \ldots, n.
            \end{align}
    \end{itemize}
    Based on Lemma \ref{bd3}, the possibilities \eqref{(2)} and \eqref{(4)} do not happen. According to Lemma \ref{bd4} and Lemma \ref{bd5}, the possibility \eqref{(1)} also does not happen. 

    Next, if $y_n(t) < 0$ for all $t \geq t_0$, by the same argument as in the proof of \eqref{bs2}, we claim that $y_{n-1}(t) < 0$ for all $t \geq t_0$. Thus, the possibilities \eqref{(3)}, \eqref{(5)} and \eqref{(6)} occur if and only if $y_n(t) > 0$ for all $t \geq 0$.

    Consider \eqref{(3)}. It follows from Lemma \ref{bd4} and Lemma \ref{bd5} that $k=2.$ Hence, $y_1(t)$, $y_2(t) > 0$, $(-1)^{i}y_i(t) > 0$ for all $t \geq t_0$, $i =3,\ldots,n.$ This combines with $y_n(t) >0$, $t>t_0$ yields $(-1)^n>0$ and thus $n$ is even. The statement \eqref{c} is checked. 

   Concerning with \eqref{(5)}, since $y_n(t) > 0$ for $t \geq t_0$, it implies that $n$  is even and thus the assertion \eqref{b} is true.

    For the case \eqref{(6)}, it is obvious that  $n$ is odd. The assertion \eqref{a} is verified.
\end{proof}
\begin{lemma}\label{bd7}
    Let $x(\cdot)$ be an eventually positive solution of \eqref{btdk}. Assume that one of the following assumptions is true.
    \begin{itemize}
        \item [(i)] For all $i=1,2,3,\ldots,n$,
        \begin{equation*}
            (-1)^{i+1}y_i(t) >0\,\,\,\text{for $t$ large enough}. 
        \end{equation*}
        \item [(ii)] For all $i=1,2,3\ldots,n$,
        \begin{equation*}
            (-1)^{i}y_i(t) >0 \,\,\,\text{for $t$ large enough},
        \end{equation*}
        \item [(iii)] $y_1(\cdot)$ is bounded. Moreover,
         \begin{equation*}
            y_1(t) > 0, y_2(t) > 0\,\,\,\text{and}\,\,\, (-1)^{i}y_i(t) >0,\; i=3,\ldots,n,\,\,\,\text{for $t$ large enough.}
        \end{equation*}
    \end{itemize}
    Then,
    \begin{equation*}
        x(t) = O(t^{\alpha-n})\,\,\,\text{as}\;t \to \infty.
    \end{equation*}
\end{lemma}
\begin{proof}
$\textup{(a)}$ Suppose that the assumption $\textup{(i)}$ is true. There exists a $t_0 > 0$ such that
    \begin{equation*}
        x(t) > 0,\; y_1(t) > 0,\; y_2(t) < 0,\; t \geq t_0.
    \end{equation*}
    From \eqref{x} and Lemma \ref{tddh}, it leads to
    \begin{align}\label{50}
        ^C D^{n-\alpha}_{0^+}y_1(t) = x(t) +a I^{\alpha-\beta}_{0^+}x(t),\;\forall t>0.
    \end{align}
    Taking $M:=\int_0^{t_0} |x(s)|ds$. For $t > t_0$, we have 
    \begin{align}\label{51}
        x(t) +a I^{\alpha-\beta}_{0^+}x(t) &= x(t) + \frac{a}{\Gamma(\alpha-\beta)}\int_0^t (t-s)^{\alpha-\beta-1}x(s)ds \nonumber\\
        & \geq x(t) + \frac{a}{\Gamma(\alpha-\beta)}\int_0^{t_0} (t-s)^{\alpha-\beta-1}x(s)ds \nonumber\\
        &\geq x(t) -\frac{a}{\Gamma(\alpha-\beta)}(t-t_0)^{\alpha-\beta-1}\int_0^{t_0}|x(s)|ds \nonumber \\
        &= x(t) -\frac{aM}{\Gamma(\alpha-\beta)}(t-t_0)^{\alpha-\beta-1}.
    \end{align}
    Define $M_1:= \int_0^{t_0} |y_2(s)|ds$. Due to $y_1(\cdot) \in AC[0,\infty)$, it follows from Lemma \ref{dhc} that, for $t >t_0$, 
    \begin{align}\label{52}
        ^C D^{n-\alpha}_{0^+}y_1(t) &= \frac{1}{\Gamma(1+\alpha-n)}\int_0^t\frac{y'_1(s)}{(t-s)^{n-\alpha}}ds \nonumber \\
        &= \frac{1}{\Gamma(1+\alpha-n)}\int_0^{t_0}\frac{y_2(s)}{(t-s)^{n-\alpha}}ds + \frac{1}{\Gamma(1+\alpha-n)}\int_{t_0}^{t}\frac{y_2(s)}{(t-s)^{n-\alpha}}ds \nonumber\\
        &\leq \frac{1}{\Gamma(1+\alpha-n)}\frac{1}{(t-t_0)^{n-\alpha}}\int_{0}^{t_0} |y_2(s)|ds \nonumber \\
        & = \frac{M_1}{\Gamma(1+\alpha-n)}\frac{1}{(t-t_0)^{n-\alpha}}.
    \end{align}
    Combining \eqref{50}--\eqref{52}, we get
    \begin{equation*}
        x(t) -\frac{aM}{\Gamma(\alpha-\beta)}(t-t_0)^{\alpha-\beta-1} \leq \frac{M_1}{\Gamma(1+\alpha-n)}\frac{1}{(t-t_0)^{n-\alpha}},\;\forall t>t_0.
    \end{equation*}
    This means that
    \begin{equation}\label{Tuan_add2}
        x(t) \leq \frac{aM}{\Gamma(\alpha-\beta)}(t-t_0)^{\alpha-\beta-1} +\frac{M_1}{\Gamma(1+\alpha-n)}\frac{1}{(t-t_0)^{n-\alpha}}, \; \forall t>t_0. 
    \end{equation}
    Notice that $\beta +1 -\alpha > n -\alpha$, from \eqref{Tuan_add2}, we conclude that there are constants $M_2 >0$ and $t_1 > t_0$ such that
    \begin{equation}\label{Tuan_add3}
        x(t) \leq \frac{M_2}{t^{n-\alpha}},\;\forall t>t_2.
    \end{equation}
    \noindent {\textup{(b)}}  Suppose that the assumption $\textup{(ii)}$ is true. We can find $t_0 > 0$ so that
    \begin{align*}
        x(t) > 0,\; y_1(t) < 0,\; y_2(t) > 0,\;y_3(t) <0,\; \forall t \geq t_0.
    \end{align*}
    Furthermore, $y_2(\cdot)$ is strictly decreasing on $[t_0,\infty)$. Thus, 
    \begin{equation*}
        -y_1(t_0) > y_1(t) - y_1(t_0) = \int_{t_0}^t y_2(s)ds \geq y_2(t)(t-t_0),\;\forall t>t_0,
    \end{equation*}
    which implies
    \begin{align}\label{53}
        y_2(t) \leq \frac{-y_1(t_0)}{t-t_0},\; \forall t > t_0.
    \end{align}
    Now, for $t > 2t_0$, we obtain the following estimates  
    \begin{align}\label{54}
        ^C D^{n-\alpha}_{0^+}y_1(t) &= \frac{1}{\Gamma(1+\alpha-n)}\int_0^t\frac{y_2(s)}{(t-s)^{n-\alpha}}ds \nonumber \\
        & = \frac{1}{\Gamma(1+\alpha-n)}\left(\int_0^{t_0} \frac{y_2(s)}{(t-s)^{n-\alpha}}ds+\int_{t_0}^{t/2}\frac{y_2(s)}{(t-s)^{n-\alpha}}ds+\int_{t/2}^t\frac{y_2(s)}{(t-s)^{n-\alpha}}ds\right) \nonumber\\
        &\leq \frac{1}{\Gamma(1+\alpha-n)}\left(\frac{1}{(t-t_0)^{n-\alpha}}\int_0^{t_0}|y_2(s)|ds + \left(\frac{2}{t}\right)^{n-\alpha}\int_{t_0}^{t/2}y_2(s)ds + I(t) \right)\nonumber\\
        &=\frac{1}{\Gamma(1+\alpha-n)}\left(\frac{M_1}{(t-t_0)^{n-\alpha}} + \left(\frac{2}{t}\right)^{n-\alpha}(y_1(t/2) -y_1(t_0)) + I(t)\right) \nonumber \\
        &\leq \frac{1}{\Gamma(1+\alpha-n)}\left(\frac{M_1}{(t-t_0)^{n-\alpha}} -y_1(t_0) \left(\frac{2}{t}\right)^{n-\alpha} + I(t)\right),
    \end{align}
    where $I(t):= \int_{t/2}^t\frac{x_2(s)}{(t-s)^{n-\alpha}}ds$. By \eqref{53},
    \begin{align}\label{55}
        I(t) &= \int_{t/2}^t\frac{y_2(s)}{(t-s)^{n-\alpha}}ds \leq \int_{t/2}^t \frac{-y_1(t_0)}{(s-t_0)(t-s)^{n-\alpha}}ds \nonumber\\
        &\leq\frac{-y_1(t_0)}{t/2-t_0}\frac{1}{\alpha-n+1}\left(\frac{t}{2}\right)^{\alpha-n+1},\; \forall t>2t_0.
    \end{align}
    According to \eqref{54} and \eqref{55}, there are $M_3 > 0$ and $t_2 > 2t_0$ satisfying
    \begin{align*}\label{56}
        ^C D^{n-\alpha}_{0^+}y_1(t) \leq \frac{M_3}{t^{n-\alpha}}, \; \forall t > t_2,
    \end{align*}
    which together with \eqref{51} leads to
    \begin{align*}
        x(t) \leq \frac{aM}{\Gamma(\alpha-\beta)}(t-t_0)^{\alpha-\beta-1} + \frac{M_3}{t^{n-\alpha}}, \;\forall t > t_2,
    \end{align*}
    and thus
    \begin{equation}\label{Tuan_add4}
        x(t) = O(t^{\alpha-n})\,\,\,\text{as}\; t \to \infty.
    \end{equation}
   \noindent {\textup{(c)}} Suppose that the assumption $\textup{(iii)}$ is true. There exits $t_0 > 0$ with
    \begin{align*}
        x(t) > 0,\; y_1(t) > 0,\; y_2(t) > 0,\;y_3(t) <0,\; t \geq t_0.
    \end{align*}
    Due to $y_1(\cdot)$ is bounded, there is a $M_4 > 0$ so that $y_1(t) < M_4$ for all $t\geq  t_0$. By the fact $y_2(\cdot)$ is positive and decreasing on $[t_0,\infty)$, we see 
    \begin{align*}
        M_4 > y_1(t) - y_1(t_0) = \int_{t_0}^t y_2(s)ds \geq y_2(t)(t-t_0),\; \forall t>t_0,
    \end{align*}
    which shows that
    \begin{align*}
        y_2(t) \leq \frac{M_4}{t-t_0},\; \forall t > t_0.
    \end{align*}
    Repeated the arguments as in $\textup{(b)}$ enables us to claim 
    \begin{equation}\label{Tuan_add5}
        x(t) = O(t^{\alpha-n})\,\,\,\text{as}\; t \to \infty.
    \end{equation}
    In short, based on \eqref{Tuan_add3}, \eqref{Tuan_add4} and \eqref{Tuan_add5}, we conclude that 
    \begin{equation*}
        x(t) = O(t^{\alpha-n})\,\,\,\text{as}\;t \to \infty.
    \end{equation*}
    The proof is complete.
\end{proof}
\begin{lemma}\label{bd8}
    Suppose that $\textup{(G1), (G2), (G3), (H), (K)}$ are true. Let $n$ be even and $x(\cdot)$ be an eventually positive solution of \eqref{btdk} satisfying \eqref{c}. Then the limit $\lim_{t \to \infty}y_1(t)$ is finite.
\end{lemma}
\begin{proof}
   It follows from the assumptions of the lemma that there is a $t_0 > 0$ such that $x(t) >0$, $y_1(t) > 0$, $y_2(t) >0$, $(-1)^iy_i(t) >0,$ $ i=3,\ldots,n$, $f(x(t)) > 0$, $q(t) \geq c_3 >0$ and $ g(t) \geq c_4t^{n-\beta-\sigma}$ for all $t \geq t_0$. From the arguments in the proof of Lemma \ref{bd4}, we can find a $t_1 > t_0 $ with
    \begin{align*}
        y_n'(t) < -q(t)f(x(t)), \; \forall t > t_1,
    \end{align*}
    which implies
    \begin{align*}
        \frac{1}{(n-2)!}&\int_{t_1}^t (s-t_1)^{n-2}q(s)f(x(s))ds < \frac{-1}{(n-2)!}\int_{t_1}^t (s-t_1)^{n-2}y'_n(s)ds \nonumber \\
        & = \frac{-(t-t_1)^{n-2}}{(n-2)!}y_n(t) + \frac{1}{(n-3)!}\int_{t_1}^t (s-t_1)^{n-3}y'_{n-1}(s)ds \nonumber \\
        &=\frac{-(t-t_1)^{n-2}}{(n-2)!}y_n(t) + \frac{(t-t_1)^{n-3}}{(n-3)!}y_{n-1}(t) - \frac{1}{(n-4)!}\int_{t_1}^t (s-t_1)^{n-4}y'_{n-2}(s)ds.
    \end{align*}
   We continue in this fashion to obtain 
    \begin{align}\label{49}
        \frac{1}{(n-2)!}&\int_{t_1}^t (s-t_1)^{n-2}q(s)f(x(s))ds < \frac{-1}{(n-2)!}\int_{t_1}^t (s-t_1)^{n-2}y'_n(s)ds \nonumber \\
        &=\frac{-(t-t_1)^{n-2}}{(n-2)!}y_n(t) + \frac{(t-t_1)^{n-3}}{(n-3)!}y_{n-1}(t)+\cdots+\frac{t-t_1}{1!}x_3(t) -\int_{t_1}^t y'2(s)ds \nonumber\\
        &=\frac{-(t-t_1)^{n-2}}{(n-2)!}y_n(t) + \frac{(t-t_1)^{n-3}}{(n-3)!}y_{n-1}(t)+\cdots+\frac{t-t_1}{1!}y_3(t) +y_2(t_1)-y_2(t)\nonumber\\
        &=\sum_{k=0}^{n-2}\frac{(-1)^{k+1}(t-t_1)^k}{k!}y_{k+2}(t) +y_2(t_1),\;t>t_1.
    \end{align}
    Because  $(-1)^{k+1}y_{k+2}(t) = - (-1)^{k+2}y_{k+2}(t)  < 0$ for $t \geq t_0,$ $k=0, 1, \ldots, n-2$, we have $\frac{(-1)^{k+1}(t-t_1)^k}{k!}y_{k+2}(t) \leq 0$ for $t\geq t_1$, $k=0,1,\ldots,n-2$. This combines with \eqref{49} leads to 
    \begin{align*}
        y_2(t_1) > \frac{1}{(n-2)!}\int_{t_1}^t (s-t_1)^{n-2}q(s)f(x(s))ds \geq \frac{c_3}{(n-2)!}\int_{t_1}^t (s-t_1)^{n-2}f(x(s))ds,\;\forall t>t_1.
    \end{align*}
    Letting $t \to \infty$ yields
    \begin{align*}
        \int_{t_1}^\infty (s-t_1)^{n-2}f(x(s))ds <\frac{(n-2)!y_2(t_1)}{c_3} < \infty.
    \end{align*}
    Choosing $M_1>0$ with
    \begin{align*}
        (t-t_1)^{n-2}f(x(t)) \leq M_1,\; \forall t \geq t_1.
    \end{align*}
    It deduces from $\textup{(G2)}$ that $x(t) < 1$ for all $t > t_2:=  \left(\frac{M_1}{c_1}\right)^{\frac{1}{n-2}}+t_1$. This together with $\textup{(G3)}$ gives
    \begin{align*}
        x(t) \leq \left(\frac{M_1}{c_2}\right)^{\frac{1}{\mu}}\frac{1}{(t-t_1)^{\frac{n-2}{\mu}}},\;\;\forall t>t_2,
    \end{align*}
    where $\mu \in \left(0,\frac{n-2}{n-\beta}\right)$. Hence, there is a $t_3 > t_2$ with
    \begin{align*}
        x(t) \leq \frac{1}{t^{n-\beta}}, \quad t\geq t_3.
    \end{align*}
    Put $M_2: = \int_0^{t_3}|x(s)|ds$. For $t >t_3$, then
    \begin{align*}
        x(t) + a I^{\alpha-\beta}_{0^+}x(t) &\leq \frac{1}{t^{n-\beta}} + \frac{a}{\Gamma(\alpha-\beta)}\int_0^{t_3} (t-s)^{\alpha-\beta-1}x(s)ds \nonumber\\
        &\hspace{3cm}+\frac{a}{\Gamma(\alpha-\beta)}\int_{t_3}^t (t-s)^{\alpha-\beta-1}x(s)ds \nonumber \\
        & \leq \frac{1}{t^{n-\beta}} + \frac{a}{\Gamma(\alpha-\beta)}(t-t_3)^{\alpha-\beta-1}\int_0^{t_3}|x(s)|ds \nonumber\\
        &\hspace{3cm}+ \frac{a}{\Gamma(\alpha-\beta)}\int_{t_3}^t (t-s)^{\alpha-\beta-1}s^{\beta-n}ds \nonumber\\ 
       &\leq  \frac{1}{t^{n-\beta}}  +\frac{aM_2}{\Gamma(\alpha-\beta)}(t-t_3)^{\alpha-\beta-1} + \frac{a}{\Gamma(\alpha-\beta)}\int_{0}^t (t-s)^{\alpha-\beta-1}s^{\beta-n}ds \nonumber\\
       &=\frac{1}{t^{n-\beta}}  +\frac{aM_2}{\Gamma(\alpha-\beta)}(t-t_3)^{\alpha-\beta-1}  \nonumber \\
       &\hspace{3cm}+\frac{a}{\Gamma(\alpha-\beta)}\int_{0}^1 t^{\alpha-\beta-1}(1-u)^{\alpha-\beta-1}t^{\beta-n}u^{\beta-n}tdu \nonumber\\
       &=t^{\beta-n}  +\frac{aM_2}{\Gamma(\alpha-\beta)}(t-t_3)^{\alpha-\beta-1} + \frac{a\Gamma(\beta-n+1)}{\Gamma(\alpha-n+1)}t^{\alpha-n}.
    \end{align*}
    From this there exist $M_3 > 0$ and $t_4 > t_3$ satisfying
    \begin{align}\label{eq51}
        ^C D^{n-\alpha}_{0^+}y_1(t)=x(t) + a I^{\alpha-\beta}_{0^+}x(t) \leq M_3t^{\alpha-n},\; \forall t \geq t_4.
    \end{align}
    Notice that for $t > t_4$, 
    \begin{align}\label{eq52}
        ^C D^{n-\alpha}_{0^+}y_1(t) &= \frac{1}{\Gamma(\alpha-n+1)}\int_0^t \frac{y_1'(s)}{(t-s)^{n-\alpha}}ds \nonumber \\
        & = \frac{1}{\Gamma(\alpha-n+1)}\int_0^{t_4} \frac{y_2(s)}{(t-s)^{n-\alpha}}ds + \frac{1}{\Gamma(\alpha-n+1)}\int_{t_4}^t \frac{y_1'(s)}{(t-s)^{n-\alpha}}ds \nonumber \\
        &\geq \frac{1}{\Gamma(\alpha-n+1)}\frac{-1}{(t-t_4)^{n-\alpha}}\int_0^{t_4}|y_2(s)|ds \nonumber \\
        &\hspace{3cm}+\frac{1}{\Gamma(\alpha-n+1)}\frac{1}{(t-t_4)^{n-\alpha}}(y_1(t) - y_1(t_4)) \nonumber \\
        &=\frac{y_1(t)-y_1(t_4)-M_4}{\Gamma(\alpha-n+1)}\frac{1}{(t-t_4)^{n-\alpha}},
    \end{align}
    where $M_4: = \int_0^{t_4}|x_2(s)|ds$. From \eqref{eq51} and \eqref{eq52}, it follows
    \begin{equation}\label{Tuan_add6}
        M_3\geq \frac{y_1(t)-y_1(t_4)-M_4}{\Gamma(\alpha-n+1)}\frac{t^{n-\alpha}}{(t-t_4)^{n-\alpha}},\;\;\forall t>t_4.
    \end{equation}
    Due to $y_1(\cdot)$ is positive and increasing on $[t_0,\infty)$. From \eqref{Tuan_add6}, we conclude that the limit $\lim_{t \to \infty}y_1(t)$
   exists and is finite.
\end{proof}
We can now prove the main results of this part.
\begin{proof}[Proof of Theorem \ref{dl3}]
    Let $x(\cdot)$ be a non-oscillatory solution of the initial value problem \eqref{btdk}--\eqref{gtd}. Without loss of generality, we assume that $x(\cdot)$ is eventually positive. From Lemma \ref{bd6}, $x(\cdot)$ satisfies \eqref{a}. Applying Lemma \ref{bd7}, we conclude $$x(t)=O(t^{\alpha-n})\;\;\text{as}\;\;t\to\infty.$$
\end{proof}
\begin{proof}[Proof of Theorem \ref{dl4}]
Let $x(\cdot)$ be a non-oscillatory solution of the initial value problem \eqref{btdk}--\eqref{gtd}. Without loss of generality, we assume that $x(\cdot)$ is eventually positive. Lemma \ref{bd6} shows that $x(\cdot)$ satisfies \eqref{b} or \eqref{c}. 
Applying Lemma \ref{bd7} and Lemma \ref{bd8}, we conclude that 
$$x(t)=O(t^{\alpha-n})\;\;\text{as}\;\;t\to\infty.$$
\end{proof}

\end{document}